\documentclass[1p,times,fleqn]{elsarticle}
\usepackage[english]{babel}
\usepackage{graphicx}
\usepackage{amsmath}
\usepackage{amsfonts}
\usepackage{amssymb}
\usepackage{hyperref}
\usepackage{lscape}
\usepackage{float}

\newtheorem{theorem}{Theorem}[section]

\newtheorem{rqm}[theorem]{Remark}

\newtheorem{assump}[theorem]{Assumption}

\newcommand{\ud}{\mathrm{d}}
\newenvironment{proof}[1][Proof]{\textbf{#1.} }{\rule{0.5em}{0.5em}}
\journal{Applied Mathematics and Computation}

\begin{document}

\begin{frontmatter}

\title{Estimation for seasonal fractional ARIMA with stable innovations via the empirical characteristic function method}


\author{Mor Ndongo $^{\sharp}$, Abdou K\^{a} Diongue $^{\sharp}$, Aliou Diop $^{\sharp}$\footnote{Corresponding author.Tel: +221 77 529 96 93, Fax: +221 33 961 53 38} Simplice Dossou-Gb\'et\'e $^{\S}$}

\address{$^{\sharp}$LERSTAD, UFR de Sciences Appliqu?es et de Technologie, BP 234, Universit\'{e} Gaston Berger, Saint-Louis, S\'en\'egal.
morndongo2000@yahoo.fr, abdou.diongue@ugb.edu.sn, aliou.diop@ugb.edu.sn.}

\address{$^{\S}$ LMA UMR CNRS 5142, BP 576, Universit\'e de Pau et des Pays de l'Adour, France. simplice.dossou-gbete@univ-pau.fr.}

\begin{abstract}
Maximum likelihood methods, while widely used, may be non-robust due to disagreement between the assumptions upon which the
models are based and the true density probability distribution of observed data. Because the Empirical Characteristic Function (ECF)
is the Fourier transform of the empirical distribution function, it retains all the information in the sample but can overcome
difficulties arising from the likelihood. This paper discusses, the ECF method proposed by Knight and Yu (2002), to estimate simultaneously the parameters for stable seasonal fractional
ARIMA processes. Under some assumptions, we show that the resulting estimators are consistent and asymptotically normally distributed.
For comparison purpose, we consider also a Two-Step Method (TSM) including in the first step the MCMC Whittle method developed by Ndongo {\it et al.} (2010), and in the second the MLE method introduced by Alvarez and Olivares (2005), to estimate the innovation parameters. The performance of the two methods is discussed under different parameter settings, using a Monte Carlo simulation.

\end{abstract}

\begin{keyword}
Seasonal Fractional ARIMA, Stable distributions, ECF estimate, Whittle estimate, Markov Chains Monte Carlo, Two-Step Method.
\end{keyword}

\end{frontmatter}
\section{Introduction}\noindent
Seasonal Fractional ARIMA time series with stable innovations were introduced by Diongue {\it et al.} \cite{DDM}. Since it is a direct generalization of the fractional ARIMA model of Kokoszka and Taqqu \cite{KT1995}, it allows one to take into account three stylized facts: long range dependence, seasonality and high variability, often encountered in financial data. The use of processes with infinite variance has received a great deal of interest in the statistical
literature as noticed in Brockwell and Davis \cite{BD2006}. Examples where such models appear to be appropriate have been found by Stuck and Kleiner \cite{SK1974}, who considered telephone signals, and Fama \cite{Fama}, who modeled stock market prices.\\

\noindent
To solve the parameter estimation problem of the stable ARFISMA processes, Ndongo {\it et al.} \cite{NDDS} propose several estimation procedures: the semiparametric method developed by Reisen {\it et al.} \cite{RRP1}, the classical Whittle estimate and the MCMC Whittle method which is based on the evaluation of the Whittle likelihood function using the Markov Chains Monte Carlo method (MCMC, in short). They study the behavior of these methods through Monte Carlo simulations, and the results show, in general, that the MCMC Whittle method is the best one. Unfortunately, this method only estimates the long-memory and short memory parameters. Moreover, for inference purposes, estimation of the innovation parameter is required. The usual response to such difficulties is that the density functions of the processes cannot be written in a closed form in the sense that it is not expressible in terms of known elementary functions. As a result, ML estimation is often very difficult (for detailed discussion, see Calder and Davis \cite{CD}). Estimation of the stable ARFISMA model via alternative methods such as Quasi-Maximum Likelihood (QML) method and Generalized Method of Moments (GMM) also presents difficulties. For example, QML is infeasible because the variance of the error term may be infinite. For GMM care must be taken when choosing moment conditions because the stable distribution does not have a finite absolute moment of order higher than the tail index $\alpha$. Consequently, another alternative is to use the Empirical Characteristic Function (ECF) method.
 \\

\noindent
A main objective of this paper is to estimate simultaneously the parameters of the symmetric $\alpha$-stable ARFISMA processes using the Empirical Characteristic Function (ECF) method. The asymptotic properties of the ECF estimators are established under some assumptions. Monte Carlo simulations are performed to study the finite sample properties of the ECF method. For comparison purpose, we consider also a Two-Step Method (TSM).\\

\noindent
The paper is organized as follows. In Section \ref{Sec.2}, we present the class of seasonal fractionally integrated processes with stable innovations. Section \ref{Sec.3} addresses the ECF method and the asymptotic properties. Section \ref{secMCMCW} reviews the Two-Step Method. Section \ref{Sec.5} illustrates the ECF procedure and compares it to the TSM in a Monte Carlo study. Section \ref{Sec.6} concludes.
\section{Model}\label{Sec.2}
\subsection{Stable distributions}
\noindent In this section, we summarize the relevant facts associated with the stable distributions and refer the reader to Samorodnitsky and Taqqu (\cite{SamTaqqu}) for a detailed statistical description. There are several ways of defining the stable distributions. The $\alpha$-stable distribution requires four parameters for complete description: an index of stability $\alpha\in (0,\ 2]$ also called the tail index, tail exponent or characteristic exponent, a skewness parameter $\beta\in [-1,\ 1]$, a scale parameter $\sigma>0$ and a location parameter $\mu\in\mathbb{R}$. The tail exponent $\alpha$ determines the rate at which the tails of the
 distribution tape off. When $\alpha=2$, the Gaussian distribution results. When $\alpha < 2$, the variance is infinite and the
 tails are asymptotically equivalent to Pareto law, i.e. they exhibit a power-law behavior. When $\alpha > 1$, the mean of the distribution
 exists and is equal to $\mu$. When the skewness parameter $\beta$ is positive, the distribution is skewed to the right and when
 it is negative, it is skewed to the left. When $\beta = 0$, the distribution is symmetric about $\mu$. As $\alpha$ approaches 2,
 $\beta$ loses its effect and the distribution approaches the Gaussian distribution regardless of $\beta$. The last two parameters,
 $\sigma$ and $\mu$ are the usual scale and location parameters, i.e. $\sigma$ determines the width and $\mu$ the shift of the mode
 of the density. In general, it will be convenient to define the $\alpha$-stable random variables in terms of their characteristic functions. A random variable $X$ is said $\alpha$-stable, denoted $S_{\alpha,\beta}(\mu,\sigma)$, if its characteristic function is given by
\begin{equation}
\label{cf}
\Phi_{X}\left(t\right) = \left\{\begin{array}{ll}
\exp\left\{i\mu t-\sigma^{\alpha}\left|t\right|^{\alpha}\left[1-i\beta\delta_t\tan\frac{\pi\alpha}{2}\right]\right\} & \textrm{if $\alpha\neq 1$,}\\\\
\exp\left\{i\mu t-\sigma\left|t\right|\left[1+i\beta\frac{2}{\pi}\delta_t\ln\left|t\right|\right]\right\} & \textrm{if $\alpha=1$,}
\end{array}\right.
\end{equation}
\begin{equation}
\mbox{where}\quad \delta_t= \left\{\begin{array}{lll}
~~1 & \textrm{if $ t > 0$,}\\
~~0 & \textrm{if $ t =0$,}\\
-1 & \textrm{if $ t < 0$.}\nonumber
\end{array}\right.
\end{equation}
Using the inverse Fourier transform of the characteristic function $\Phi_X(t)$, we can give a integral representation of the probability density function as:
\begin{equation}
f\left(x;\alpha,\sigma,\beta,\mu\right)=\frac{1}{2\pi}\int^{+\infty}_{-\infty}\exp\left(-itx\right)\Phi_X\left(t\right)\ \ud t.
\end{equation}
Then the probability density function of a symmetric standard $\alpha$-stable random variable (i.e. $\sigma=1$, $\mu=0$ and $\beta=0$) can be expressed as:
\begin{equation}\label{intdens}
f\left(x;\alpha\right)=\frac{1}{\pi}\int^{+\infty}_{0}\exp\left(-\left|t\right|^{\alpha}\right)\cos\left(tx\right)\ \ud t.
\end{equation}
Formula (\ref{intdens}) does not have closed form expression, except in three cases (Levy, Cauchy and Gaussian distributions). However, it can be numerically integrated.\\

\noindent In the following of this paper, we will consider the standard symmetric $\alpha$-stable distribution that we denote S$\alpha$S (i.e. the case where $\beta=0$, $\mu=0$ and $\sigma=1$).
\subsection{Stable ARFISMA model}\label{section22}\noindent
Seasonal Fractional Autoregressive Integrated Moving Average time series with symmetric $\alpha$-stable (S$\alpha$S) innovations, denoted hereafter by ARFISMA-S$\alpha$S, were studied in Diongue {\it et al.} \cite{DDM}. These models exhibit long range dependence, seasonality and high variability, and are an infinite variance counterpart to the ARFISMA model introduced by Reisen {\it et al.} \cite{RRP1}. In this Section, we examine the definition and the basic characteristics of stable ARFISMA model, and we refer to Diongue {\it et al.} \cite{DDM} or Ndongo {\it et al.} \cite{NDDS}  for a detailed description.\\
 
\noindent Suppose $\left(Z_t\right)_{t\in\mathbb{Z}}$ is
a sequence of independently and identically distributed (i.i.d.) S$\alpha$S ($0<\alpha\leq 2$) random variables with mean zero  and scale parameter equals to 1. Let $B$ be the back shift operator and $s$ the seasonal period, then the polynomials of orders $p,\ q,\ P,\ Q$  are respectively defined by:
$$\phi_p(B)=1-\phi_1B-\phi_2B^2-\cdots-\phi_pB^p\qquad \theta_q(B)=1+\theta_1B+\theta_2B^2+\cdots+\theta_qB^q$$
$$\Phi_P(B^s)=1-\Phi_{s}B^{s}-\Phi_{2s}B^{2s}-\cdots-\Phi_{Ps}B^{Ps}\qquad \Theta_Q(B^s)=1+\Theta_sB^{s}+\Theta_{2s}B^{2s}+\cdots+\Theta_{Qs}B^{Qs}.$$
It is assumed that these polynomials have no common zeros and satisfy the conditions
$\Phi(z^s)\phi(z)\neq 0$ and $\Theta(z^s)\theta(z)\neq 0$ for $|z|=1$. Futhermore, in the above equations, $(\Phi_i)_{1\leq i\leq P}$, $(\phi_j)_{1\leq j\leq p}$, $(\Theta_k)_{1\leq k\leq Q}$ and $(\theta_l)_{1\leq l\leq q}$ are unknown parameters.\\

\noindent A zero-mean process $(X_t )_{t\in\mathbb{Z}}$ is said a seasonal fractionally integrated process with S$\alpha$S innovations, denoted here by ARFISMA(p, d, q)$\times$(P, D, Q)$_{s}$-S$\alpha$S, if the following equation is satisfied
\begin{equation}
\label{model}
\phi_p (B)\Phi_P (B^s)X_t =(1-B)^{-d}(1-B^{s})^{-D}\theta_q (B)\Theta_Q(B^s)Z_t,
\end{equation}
where the long-memory parameters $d$ and $D$ are fractional parameters. Notice that this model is a direct generalization of the ARFIMA-S$\alpha$S model of Kokoszka and Taqqu \cite{KT1995} and it contains several particular cases (e.g. Diongue {\it et al.} \cite{DDM} for more details). We assume that the following condition holds:
\begin{equation}\label{cond1}
\left|d+D\right|<1-\frac{1}{\alpha}\ \textrm{and}\ \left|D\right|<1-\frac{1}{\alpha},\quad\textrm{with}\  1<\alpha\leq 2.
\end{equation} 
\noindent
According to Giraitis and Leipus \cite{GL} or Reisen et al. \cite{RRP1}, one can easily show that
{\setlength\arraycolsep{2pt}
\begin{eqnarray}\label{seasonal}
(I-B)^{d}(I-B^s )^{D} &= &\displaystyle\prod_{j=0}^{[\frac{s}{2}]}\big[(1-e^{i\lambda_j }B)(1-e^{-i\lambda_j }B)\big]^{d_j}{}
\nonumber\\
&=&\displaystyle\prod_{j=0}^{[\frac{s}{2}]}(1-2cos\lambda_j B + B^2 )^{d_j},
\end{eqnarray}}
\noindent
with $d_0 = \frac{d+D}{2}$, $d_i = D$, for $i=1,\ldots, [\frac{s}{2}]-1$, $d_{[\frac{s}{2}]}=\frac{D}{2}$, and $\lambda_j =\frac{2\pi j}{s}$, for $j=0,\ldots, [\frac{s}{2}]$.\\
 By means of the expansion
$$\displaystyle\prod_{j=0}^{[\frac{s}{2}]}(1-2cos\lambda_j B + B^2 )^{d_j}=\sum_{j=0}^{+\infty}b_{j}(d,\ \nu)B^j,$$
where the coefficients $\left(b_j(d,\ \nu)\right)_{j\geq 0}$ are given by:
\begin{equation}\label{bj}
b_j(d,\ \nu) =\sum_{\substack{0\leq l_0,\cdots,l_{[\frac{s}{2}]}\leq j,\\ l_0+\cdots+l_{[\frac{s}{2}]}=j}}C_{l_0}\left(d_0,\nu_0\right)\cdots C_{l_{[\frac{s}{2}]}}\left(d_{[\frac{s}{2}]},\nu_{[\frac{s}{2}]}\right),
\end{equation}
where $d=(d_0,\ldots,d_{[\frac{s}{2}]})$, $\nu=(\nu_0,\ldots,\nu_{[\frac{s}{2}]})$ with $\nu_j=\cos(\lambda_j)$, for  $j=0,\ldots, [\frac{s}{2}]$.
\noindent The weights $\left(C_{l}\left(d_i,\nu_i\right)\right)_{l\in\mathbb{Z}}$ are the Gegenbauer polynomials and they can be computed using the following recursion formula:
\begin{displaymath}
\left\{ \begin{array}{l}
C_{0}\left(d_{i},\nu_{i}\right)= 1\\
C_{1}\left(d_{i},\nu_{i}\right)= 2d_{i}\nu_{i}\\
C_{j}\left(d_{i},\nu_{i}\right)=
2\nu_{i}\left(\frac{d_{i}-1}{j}+1\right)C_{j-1}\left(d_{i},\nu_{i}\right)-\left(2\frac{d_{i}-1}{j}+1\right)C_{j-2}\left(d_{i},\nu_i\right),\ \forall j>1.
\end{array} \right.
\end{displaymath}
Hence the process defined by (\ref{model}) can be rewritten as:
$$\Phi_P (B^s)\phi_p (B)X_t =\Theta_Q(B^s)\theta_q (B)\displaystyle\prod_{j=0}^{[\frac{s}{2}]}(1-2\cos(\lambda_j) B + B^2 )^{-d_j} Z_t.$$
Therefore, under conditions (\ref{cond1}) Diongue {\it et al.} \cite{DDM} show that the process 
$\left(X_t\right)_{t\in\mathbb{Z}}$ is stationary and invertible, and the $AR(\infty)$ and $MA(\infty)$ representations are respectively given by:
\begin{equation}\label{ARrepres}
Z_t=\frac{\Phi_P\left(B^s\right)\phi_p\left(B\right)}{\Theta_Q\left(B^ s\right)\theta_q\left(B\right)}\displaystyle\prod_{j=0}^{[\frac{s}{2}]}(1-2\nu_j B + B^2 )^{d_j} X_t = \sum^{\infty}_{j=0}\tilde{c}_jX_{t-j}
\end{equation}
and
\begin{equation}\label{MArepres}
X_t=\frac{\Theta_Q\left(B^ s\right)\theta_q\left(B\right)}{\Phi_P\left(B^s\right)\phi_p\left(B\right)}\displaystyle\prod_{j=0}^{[\frac{s}{2}]}(1-2\nu_j B + B^2 )^{-d_j} Z_t = \sum^{\infty}_{j=0}c_jZ_{t-j}.
\end{equation}
The coefficients $\left( \tilde{c}_j\right)_{j\geq 0}$ and $\left( c_j\right)_{j\geq 0}$ are defined respectively by:
\begin{equation}
\label{coefCjtilde}
\Phi_{P}(z^s)\phi_{p}(z)\sum_{j=0}^{+\infty}\pi_j(d,\ \nu)z^j=\Theta_{Q}(z^s)\theta_{q}(z)\sum_{j=0}^{+\infty}\tilde{c}_{j}z^{j},\qquad \textrm{for}\ |z|<1
\end{equation}
and
\begin{equation}
\label{coef}
\Theta_{Q}(z^s)\theta_{q}(z)\sum_{j=0}^{+\infty}b_{j}(d,\ \nu)z^j=\Phi_{P}(z^s)\phi_{p}(z)\sum_{j=0}^{+\infty}c_{j}z^{j},\qquad \textrm{for}\ |z|<1,
\end{equation}
where the weights $\left( \pi_j(d,\ \nu)\right)_{j\geq 0}$ are such that $\pi_j(d,\ \nu)= b_j(-d,\ \nu)$, with the coefficients $\left(b_j(-d,\ \nu)\right)_{j\geq 0}$ given in equation (\ref{bj}). In the particular case where $P=Q=0$ \footnote{Expressions of these coefficients for others values of $P$ and $Q$ can be found in Ndongo \cite{Ndongo2011}}, It is easy to verify that the coefficients $\left( \tilde{c}_j\right)_{j\geq 0}$ and $\left( c_j\right)_{j\geq 0}$ can be computed using the following recursion formula: 
\begin{equation}\label{Cjtilde}\tilde{c}_0=1\quad \textrm{and}\quad
\tilde{c}_j=\pi_j\left(d,\nu\right)-\sum^{\min\left(j,p\right)}_{i=1}\phi_i\pi_{j-i}\left(d,\nu\right)+\sum^{\min\left(j,q\right)}_{i=1}\theta_i\tilde{c}_{j-i},\qquad \forall j\geq 1
\end{equation}
and
\begin{equation}\label{cj} c_0=1\quad \textrm{and}\quad
c_j=b_j\left(d,\nu\right)+\sum^{\min\left(j,p\right)}_{i=1}\phi_{i} c_{j-i}-\sum^{\min\left(j,q\right)}_{i=1}\theta_i b_{j-i}\left(d,\nu\right),\qquad \forall j>1.
\end{equation}
The power transfer function of $\left(X_t\right)_{t\in\mathbb{Z}}$ is given by:
\begin{equation}
 \label{ftr}
  h_X\left(\lambda\right)=\left|\frac{\theta_q\left(e^{-i\lambda}\right)\Theta_Q\left(e^{-i\lambda s}\right)}{\phi_p\left(e^{-i\lambda}\right)\Phi_P\left(e^{-i\lambda s}\right)}\right|^2\left|2\sin\left(\frac{\lambda}{2}\right)\right|^{-2d}\left|2\sin\left(\frac{\lambda s}{2}\right)\right|^{-2D},\  -\pi\leq\lambda\leq\pi.
 \end{equation}
\section{ECF method and asymptotic properties}\label{Sec.3}
\subsection{ECF method}\noindent
Because the Empirical Characteristic Function (ECF) is the Fourier transform of the empirical distribution function, it retains all the information in the sample
but can overcome difficulties arising from the likelihood. The basic idea for the ECF method is to minimize some distance measure between the empirical characteristic function and the Characteristic Function (CF). It should be noted that, there are various ECF methods and the approach in the dependent case is not exactly the same as in the i.i.d case, because the dependence must be taken into account. In this Section, we summarize the approach in the dependent case and refer the reader to Knight and Yu \cite{KY} or Yu \cite{yu}, for more details.\\

\noindent
Given a stationary symmetric $\alpha$-stable ARFISMA process $\left(X_t\right)_{t\in\mathbb{Z}}$ defined by equation (\ref{model}). We denote by
$\psi=(\alpha,\ d,\ D,\ \phi_{1},\ldots,\phi_{p},\ \theta_{1},\ldots,\theta_{q},\ \Phi_{1},\ldots,\Phi_{P},\ \Theta_{1},\ldots,\Theta_{Q})$, its parameters. The preceding discussion in Section \ref{section22} motivates the choice of our parameter space $\Psi$ given by
$$\Psi=\{\psi\in\mathbb{R}^{p+q+P+Q+3} : \phi_{p}\neq 0,\ \theta_{q}\neq 0,\ \Phi_{P}\neq 0,\ \Theta_{Q}\neq 0, \theta_q\left(z\right)\ ,\ \phi_p\left(z\right),\ \Phi_p\left(z^s\right)\ \textrm{and}\ \Theta_Q\left(z^s\right),$$
 $$ \textrm{have no common zeros for}\ \left|z\right|\leq 1,\ \left|d+D\right|<1-\frac{1}{\alpha}\ \textrm{and}\ \left|D\right|<1-\frac{1}{\alpha},\ 1<\alpha\leq 2 \}.$$
 We assume that $\psi_0$ is the true value of $\psi$ and that $\psi_0$ is in the interior of the compact set $\Psi\subseteq \mathbb{R}^{p+q+P+Q+3}$.
We wish to estimate $\psi$ from a finite realization $\left(X_1,\ldots,X_T\right)$ by the ECF method. In the dependent case, the procedures involve moving blocks of data. Denote the moving blocks for $\left(X_1,\ldots,X_T\right)$ as $Y_{j}=\left(X_{j},\ldots,X_{j+m}\right)^{'},\ j=1,\ldots,T-m$. Thus each block has $m+1$ observations.
\begin{enumerate}
\item[$\bullet$] The CF of each block is defined by:
 $$c\left(r,\psi\right)=\mathbb{E}\left(\exp\left(ir'Y_j\right)\right),\ \textrm{where}\ r=(r_1,\ldots,r_{m+1})'\ \  \textrm{is the transform variables}.$$
\item[$\bullet$] The joint ECF is given by:
 $$c_{n}\left(r\right)=\frac{1}{n}\sum_{j=1}^{n}\exp\left(ir'Y_j\right),\ \textrm{where}\ n=T-m.$$
\item[$\bullet$] To estimate the parameter $\psi$ via the ECF method, one can minimize
 \begin{equation}
 \label{eqI1}
 I_{n}(\psi)=\int\ldots\int\left|c_{n}(r)-c(r;\psi)\right|^{2}g(r)\ \ud r,
 \end{equation}
 where $g(r)$ being a continuous weighting function. Or equivalently, one can minimize
 \begin{equation}\label{const}
 I_{n}(\psi)=\int\ldots\int\left|c_{n}(r)-c(r;\psi)\right|^{2}\ \ud G(r),
 \end{equation}
 or solve the following estimation equation
 \begin{equation}
 \label{wf}
 \int\ldots\int\left(c_{n}(r)-c(r;\psi)\right)W\left(r,\psi\right)\ \ud r=0,
 \end{equation}
 where $G(r)$ and $W(r,\psi)$ are weighting functions.
 \end{enumerate}
To use the ECF method, we have to ensure, at first, that the joint CF of the stable ARFISMA model has a close form. On the other hand, we need to determine the value of the size of moving blocks $m$ and the weighting function to be used. Thus, we show in theorem below, that the joint CF has close form expression.
 \begin{theorem}\label{Th3.1}\rm
 Let $\left(X_t\right)_{t\in\mathbb{Z}}$ be an ARFISMA-S$\alpha$S process defined by equation (\ref{model}). The joint CF of  $X_{t-m},X_{t-m+1},\ldots,X_{t}$ is given by:
$$c\left(r_{1},r_{2},\ldots,r_{m+1};\psi\right)=\exp\left\{-\sum_{j=0}^{\infty}\left|\sum_{l=0}^{m}r_{l+1}c_{j+l}\right|^{\alpha}-\sum_{l=2}^{m+1}\left|\sum_{h=0}^{m+1-l}r_{h+l}c_{h}\right|^{\alpha}\right\},$$
where the coefficients $\left(c_j\right)_{j\in\mathbb{Z}}$ are given by equation (\ref{coef}).
 \end{theorem}
 \begin{proof} The proof is postponed to Appendix A. \end{proof}\\
 
 \noindent
 Concerning the choice of the weighting function, according to Knight and Yu \cite{KY} or Yu \cite{yu} one can obtain a optimal weighting function $W^{*}(r,\psi)$ given by:
\begin{equation}
W^{*}\left(r,\psi\right)=\int\ldots\int\exp\left(ir'Y_j\right)\times\frac{\partial\log f\left(X_{j+m}|X_{j},\ldots,X_{j+m-1}\right)}{\partial\psi}\ \ud X_{j}\ldots\ud X_{j+m},
\end{equation}
where $f\left(X_{j+m}|X_{j},\ldots,X_{j+m-1}\right)$ is the conditional score function. This weight is optimal in the sense that the asymptotic variance of the estimator based on equation (\ref{wf}) can be made arbitrarily close to the Cram\'er-Rao lower bound when $m$ is large enough. Because the conditional score of the stable ARFISMA model is unknown, these procedure is not feasible. Instead, in this paper, we use the sub-optimal ECF method with an exponential weight $g(r)=\exp(-r'r)$. The advantages of using an exponential weight are twofold. First, it puts more weight on the interval around the origin, consistent with the recognition that the CF contains the most information around the origin and the second reason is for computational convenience.\\

\noindent
It is very important to recognize that when using the joint ECF, an additional choice needs to be made, which is that of the overlapping size of the moving blocks, $m$. The efficiency of the ECF estimator is very sensitive of the choice of $m$, as the moving blocks with a different $m$ may contain different amounts of information in the sample. In general, there is a trade-off between large $m$ and small $m$. Since the process in (\ref{model}) does not have a Markov property, a large value of $m$ works better than the smaller $m$ in term of the asymptotic efficiency, because it retains all the information in the moving blocks of the original sequence. Unfortunately, it is not very clear how $m$ affects efficiency, and we do not have an obvious guide to the choice of $m$. Ideally, $m$ should be selected so that it minimizes the mean-square error (MSE) of the ECF estimator. Precisely, let $\hat{\psi}^{(m)}$ be the ECF estimator for a given $m$. Then we look for $m_{opt}$ such that
\begin{equation}
\label{crit}
m_{opt}=\arg\min_{m}\frac{1}{R}\sum_{j=1}^{R}\left(\hat{\psi_j}^{(m)}-\psi_0\right)^2,
\end{equation}
where $R$ is the number of replication, $\psi_0$ is the nominal value of $\psi$ and $\hat{\psi_j}^{(m)}$ is the estimate for sample $j$.
\subsection{Asymptotic properties}\noindent
The asymptotic properties of the ECF estimator in the i.i.d case have been obtained by Heathcote \cite{Heathcote1977}. In the dependent case, Knight and Yu \cite{KY} establish, under some regularity conditions, that the ECF estimator is consistent and asymptotically normally distributed. They give sufficient conditions to check the conditions listed by Newey and McFadden \cite{NMc1994}, namely compactness, continuity, uniform convergence, and identifiability, under which an extremum estimator is consistent and asymptotically normally distributed. In this Section, our work consist to check a set of sufficient assumptions to verify regularity conditions listed in Knight and Yu \cite{KY}. So for any fixed $m$, the following assumptions are necessary:
\begin{assump}\label{Assump3.2}\rm
The parameter space $\Psi\subset\mathbb{R}^{p+q+P+Q+3}$ is a compact set with $\psi_0$ is in the interior of $\Psi$.
\end{assump}

\begin{assump}\label{Assump3.3}\rm
The coefficients $\left(c_j\right)_{j\geq 0}$ in the moving average representation (\ref{MArepres}) satisfy the following properties:
\begin{enumerate}
\item[(a)] $\quad\displaystyle\sum_{j=0}^{\infty}\left|c_j\right|^{2(\alpha-1)}<\infty,\qquad\forall\  1<\alpha\leq 2$,

\item[(b)] $\quad\displaystyle\sum_{j=0}^{\infty}\left|\frac{\partial~c_{j}}{\partial\delta}\right|^{2}<\infty\quad\textrm{and}\quad \displaystyle\sum_{j=0}^{\infty}\left|\frac{\partial^2~c_{j}}{\partial\delta\partial\delta'}\right|^{2}<\infty,$

\item[(c)] $\quad\displaystyle\sum_{j=0}^{\infty}\left(\frac{\partial}{\partial\delta}\log\left|\sum_{l=0}^{m}r_{l+1}c_{j+l}\right|\right)^4<\infty$,
\end{enumerate}
where $\delta=(d,\ D,\ \phi_{1},\ldots,\phi_{p},\ \theta_{1},\ldots,\theta_{q},\ \Phi_{1},\ldots,\Phi_{P},\ \Theta_{1},\ldots,\Theta_{Q}).$
\end{assump}

\begin{assump}\label{Assump3.4}\rm
$\quad\displaystyle\int\ldots\int \left|r_j\right|^6g(r)\ dr<\infty,\ \forall\ j=1,\ldots,m+1.$
\end{assump}

\begin{assump}\label{Assump3.5}\rm Let
$I_{0}(\psi)=\displaystyle\int\ldots\int\left|c(r;\psi_0)-c(r;\psi)\right|^{2}g(r)\ dr$ and $I_{0}(\psi)=0$ only if $\psi=\psi_0$.
\end{assump}

\begin{assump}\label{Assump3.6}\rm
 Let $\Im_j$ be a $\sigma$-algebra such that $\left\{K_j,\Im_j\right\}$ is an adapted stochastic sequence, where $K_j=K(Y_j;\ \psi)$ and $K(x;\ \psi)$ is a function defined by:
$$K(x;\ \psi)=\int\ldots\int\left(\cos(r'x)-c(r;\psi)\right)\frac{\partial}{\partial\psi}~c(r;\psi) g(r)\ dr.$$
We can think of $\Im_j$ as being the $\sigma$-algebra generated by the entire current and past history of $K_j$. Let $\nu_j=\mathbb{E}\left[K_0|K_{-j},K_{-j-1},\ldots\right]-\mathbb{E}\left[K_0|K_{-j-1},K_{-j-2},\ldots\right]$, for $j\geq 0$. Assume that $\mathbb{E}\left[K_0|\Im_{-l}\right]$ converges in mean square to $0$ as $l\rightarrow \infty$ and  $\sum_{j=0}^{\infty}\mathbb{E}\left[\nu_{j}^{'}\nu_j\right]<0$.
\end{assump}

\begin{assump}\label{Assump3.7}\rm $g\left(r\right)$ is a probability density function in $\mathbb{R}^{m+1}$.

\end{assump}

\begin{rqm}~\\\rm\vspace{-0.5cm}
\begin{enumerate}
\item[$\bullet$] Assumptions \ref{Assump3.3} and \ref{Assump3.4} are technical assumptions and allow to verify the hypothesis ($A_2$), ($A_3$), ($A_5$) and ($A_6$) of Knight and Yu \cite{KY}.
\item[$\bullet$] Assumptions \ref{Assump3.2}, \ref{Assump3.5}, \ref{Assump3.6} and \ref{Assump3.7} are respectively similar to assumptions ($A_1$), ($A_4$), ($A_7$) and ($A_8$) of Knight and Yu \cite{KY}. Assumption \ref{Assump3.2} ensures the compactness of the parameter set. Assumption \ref{Assump3.5} is the identification condition whereas Assumption \ref{Assump3.6} provides sufficient conditions for a strong law of large numbers and a central limit theorem for a strictly stationary and ergodic sequence.
\end{enumerate}
\end{rqm}
The asymptotic properties of the ECF estimator based on the joint characteristic function are summarize in the following theorem:

\begin{theorem}\label{Th3.9}
\label{Thm_ConsNormECF}\rm Let $\widehat{\psi}_n=\displaystyle\arg\min_{\psi\in\Psi}I_{n}(\psi)$. Suppose that Assumptions \ref{Assump3.2}-\ref{Assump3.5} and Assumption \ref{Assump3.7} hold, then $\quad\widehat{\psi}_n\stackrel{p.s}{\longrightarrow}\psi_0$. If, in addition Assumption \ref{Assump3.6} holds, then:
$$\sqrt{n}\left(\widehat{\psi}_n-\psi_0\right)\stackrel{\mathcal{D}}{\longrightarrow}\mathcal{N}\left(0,\ B^{-1}(\psi_0)A(\psi_0)B^{-1}(\psi_0)\right),$$
$$\mbox{where}\quad A(\psi_0)=Var\left(K(Y_1;\ \psi_0)\right)+ 2\sum_{j=2}^{\infty} cov\left(K(Y_1;\ \psi_0),\ K(Y_j;\ \psi_0)\right)$$
and $B^{-1}(\psi_0)$ is the inverse of the matrix $B(\psi_0)$ defined by:
$$ B(\psi_0)=\int\ldots\int\frac{\partial}{\partial\psi}~c(r;\psi_0)\frac{\partial}{\partial\psi'}~c(r;\psi_0)g(r)\ dr.$$
\end{theorem}
\begin{proof} The proof of this theorem is given in Appendix B. 
 \end{proof}
\section{The Two-Step Method}\label{secMCMCW}\noindent
In this section, the Two-Step Method (TSM) is used to estimate the parameters of ARFISMA-S$\alpha$S model. Suppose that $X_1,\ldots,X_T$ are generated by the ARFISMA-S$\alpha$S process defined by equation (\ref{model}). Denote by $\zeta=(d,\ D,\ \phi_{1},\ldots,\phi_{p},\ \theta_{1},\ldots,\theta_{q},\ \Phi_{1},\ldots,\Phi_{P},\ \Theta_{1},\ldots,\Theta_{Q})$, and $\psi=(\alpha,\ \zeta)$ the parameter vector to be estimated.    In the first step, we use the MCMC Whittle method developed in Ndongo {\it et al.}\cite{NDDS} to estimate the parameter vector $\zeta$. In the second phase, the MLE method given in Alvarez and Olivares \cite{AO2005} can be employed in the filtered series to estimate the innovation parameter $\alpha$. Therefore, the Two-Step Method can be described as follows. 
\subsection{The MCMC Whittle method}\noindent
This method is based on approximation of the Whittle likelihood function, using Markov Chains Monte Carlo (MCMC) method. Thus, the MCMC-Whittle's estimator $\hat{\zeta}_W$ of $\zeta$ is obtained by minimizing the following likelihood function:
\begin{equation}\label{eqLW}
L_{W}\left(X,\ \zeta\right)= \frac{1}{N}\sum_{j=1}^{N}\frac{1}{h_X(\lambda_j,\zeta)},
\end{equation}
where $N$ is taken large enough from the strong law of large number and $h_X(\lambda,\zeta)$ is the power transfer function of the process $\left(X_t\right)_{t\in\mathbb{Z}}$ generating the data, and is defined by equation (\ref{ftr}). In equation (\ref{eqLW}), the sequence $\lambda_{1},\ldots,\lambda_{N}$ is generated using a Metropolis-Hastings  algorithm (see Ndongo {\it et al.}\cite{NDDS} for more details).
\subsection{The EML method}\noindent
In this second step, we firstly use the estimation obtained by MCMC-Whittle method to calculate the filtered series $(Z_t)_{t\in\mathbb{Z}}$, thanks to the autoregressive representation (\ref{ARrepres}). However, given the observation $X_1,\ldots,X_T$ the innovation $\left(Z_t\right)_{t=1,\ldots,T}$ cannot be directly computed, since an infinite sample would be needed. Nevertheless, they may be estimated by:
$$Z_t\left(\hat{\zeta}_W\right)=\sum_{j=0}^{T-1} \tilde{c}_j(\hat{\zeta}_W)X_{t-j},\qquad t=1,\ldots,T,$$
where the coefficients $\left( \tilde{c}_j(\hat{\zeta}_W)\right) _{j\geq 0}$ are defined by equation (\ref{coefCjtilde}). Secondly, we estimate the innovation parameter $\alpha$ given the observations $Z_1,\ldots,Z_T$ by using a MLE procedure developed in Alvarez and Olivares \cite{AO2005}. Therefore, $\hat{\alpha}$ is obtained by maximizing the following log-likelihood function:
$$l(\alpha) = \sum_{t=1}^{T}\log\left( f(Z_t,\alpha)\right),$$
where $f(Z_t,\alpha)$ is the density probability function for the random variable $(Z_t)_{t\in\mathbb{Z}}$.    
\section{Monte Carlo experiments}\label{Sec.5}\noindent
The Monte Carlo study is designed to check variability of the ECF procedure in comparison with the Two-Step Method. The simulation results give the average values, the root mean square error (RMSE) and the mean absolute error (MAE) of the estimation procedures based on $1500$ replications. The number of observations is set at $T=1500$ and the seasonal period is fixed at $s=4$. All calculations were carried out using an R programming environment (see \cite{Rprog}) on a Intel(R) Pentium Dual-Core Processor T2050, 1.60 GHz (2 CPUs) computer.\\

\noindent
  We generate different ARFISMA sequences with S$\alpha$S innovations. We adopt the method developed by Stoev and Taqqu \cite{stotaq} for FARIMA time series with S$\alpha$S innovations. We approximate the path $X_t,\ t=1,\ldots,T\ $ by the truncation moving average
  \begin{equation}
  \label{inftrun}
X_t = \sum_{j=0}^{M}c_jZ_{t-j},
\end{equation}
where the non random constants $(c_j)_{j\geq 0}$ are defined by equation (\ref{coef}), and $M$ is the truncation parameter and fixed in this study to $5000$. Indeed, in order to represent well the long-term dependence behaviour, Stoev and Taqqu \cite{stotaq} suggest to use large values for the truncation parameter. They show also that the choice of $M$ greater or comparable to the sample $T$ work well in most case. The S$\alpha$S innovations are obtained through a version of the Chambers {\it et al.} \cite{chaMaSt} algorithm described in Section 1.7 of Samorodnitsky and Taqqu \cite{SamTaqqu}. All of the simulations involve one of  the following models summarized in Table \ref{Tabmodel}.
\begin{table}[H]
\begin{tabular}{p{0.5cm} p{2.5cm}*{5}{p{1.8cm}}}
\hline
&\textrm{Parameters}&$\alpha$&$d$&$D$&$\phi$&$\theta$\\
\hline
 &\textrm{Model 1}&$1.6$&$0.15$&$0.20$&$0$&$0$\\
&\textrm{Model 2}&$1.6$&$0.15$&$0.20$&$0.60$&$0$\\
&\textrm{Model 3}&$1.6$&$0.15$&$0.20$&$0$&$0.40$\\
&\textrm{Model 4}&$1.6$&$0.15$&$0.20$&$0.60$&$0.40$\\
\hline
\end{tabular}
\caption{Data generating processes.}
\label{Tabmodel}
\end{table}\noindent
For each model, we choose, at first, several values for the size of the moving blocks $m$ ($m=1$ to $m=8$) to examine the effect of $m$ on the estimates. Then, we select the optimal $m$ for each model using the criterion described by equation (\ref{crit}), and we compare the results with the Two-Step Method. On the other hand, we vary the values of short memory parameters, to study their effects on the estimation of long memory parameters. We show here how to compute the ECF estimator in practice for $m=1$, and the procedure is the same for other values of $m$. Note that with $m=1$, the characteristic function and the empirical characteristic function are respectively given by:
\begin{equation}\label{crpsim1}
 c(r_1,r_2;\psi)=\exp\left(-\sum_{j=0}^{\infty}\left|r_1c_j+r_2c_{j+1}\right|^{\alpha}-\left|r_2\right|^{\alpha}\right),
 \end{equation}
and
 \begin{equation}\label{cnrm1}
  c_n(r_1,r_2)=\frac{1}{n}\sum_{j=1}^{n}\exp\left(ir_1 X_j+ir_2 X_{j+1}\right),\quad n=T-1.
 \end{equation}
In this paper, we use the ECF method with an exponential weighting function, then the procedure is to choose $\psi$ to minimize
\begin{equation}
 \label{eqI1m1}
 \int\int\left|c_{n}(r_1,r_2)-c(r_1,r_2;\psi)\right|^{2}\exp\left(-r_{1}^{2}-r_{2}^{2}\right)\ \ud r_1\ud r_2,
 \end{equation}
where $c(r_1,r_2;\psi)$ and $c_n(r_1,r_2)$ are given by (\ref{crpsim1}) and (\ref{cnrm1}). The implementation of the ECF method essentially requires minimizing (\ref{eqI1m1}), which involves double integrals. Consequently, we use a Monte Carlo Integration method to evaluate the multiple integral (\ref{eqI1m1}).
\subsection{Empirical choice of the size of the moving blocks}\noindent
In this Section, we show the effects of the size of moving blocks $m$ and, then we determine the optimal value of $m$ suitable for each model presented in Table \ref{Tabmodel}. The simulation results are given in Tables \ref{tabmodel1ecf}, \ref{tabmodel2ecf}, \ref{tabmodel3ecf} and \ref{tabmodel4ecf}.
\begin{itemize}
\item In the first experiment, a symmetric $\alpha$-stable ARFISMA(0, d, 0)(0, D, 0) model is considered (i.e the model without short memory parameters), and the results are summarized in Table \ref{tabmodel1ecf}. We see that, for all the values of $m$, the ECF method performs very well as the RMSE and the MAE are small. However, this performance varies with the value of $m$. Indeed, for the long-memory parameters $d$ and $D$, the best estimates are obtained for $m=1$ whereas for the innovation parameter $\alpha$ the value $m=4$ is the best one. Nevertheless, in terms of RMSE and MAE the difference between $m=1$ and $m=4$ for the parameter $\alpha$ is relatively small (respectively $1.11\%$ and $1.18\%$). Therefore, we consider the value $m=1$ as the optimal value of the size of moving blocks to estimate simultaneously all the parameters of Model 1.
\item In the second experiment, we introduce in the previous model a short memory part, and the results are displayed in Tables \ref{tabmodel2ecf}, \ref{tabmodel3ecf} and \ref{tabmodel4ecf}. Firstly, in these tables we observe that the RMSE and the MAE of the estimates are larger for all values of $m$ compared to those obtained in Table \ref{tabmodel1ecf}. Consequently, the AR and MA components can also be a source of the bias of the estimators. This phenomenon is already observed in the literature (e.g. Boutahar {\it et al.}\cite{BMN2007}, Diongue
and Gu\'egan \cite{DG2008}). On the other hand, the results are very bad for $m=1$, $2$, $3$ and $4$, and we observe a significant improvement from $m=5$. In general, we remark that the value $m=6$ seems to give better performance. Hence, we consider the value $m=6$ as the optimal value of the size of moving blocks to estimate simultaneously all the parameters of the models with short memory components.
\end{itemize}

\noindent In view of the simulation results, we can conclude that the efficiency of the ECF method depend on the choice of the size of moving blocks. Hence, we need to determine the optimal value of $m$. Thus, for the stable ARFISMA(0, d, 0)(0, D, 0) model, i.e. without short memory components, the value $m=1$ is the optimal value of $m$ whereas when the model has a short memory part the value $m=6$ is the best one. In the Section below, the estimates with these optimal values of $m$ are compared to those obtained by the Two-Step Method (TSM).
\begin{table}[H]
\begin{center}
\begin{tabular}{p{0.5cm} p{2.5cm} l  l l p{0.3cm} l l l l l p{0.5cm}}
	\hline
  &&
	 \multicolumn{2}{l}{$m=1$}&&&&&\multicolumn{2}{l}{$m=2$}&\\
	 \cline{3-5}\cline{9-11}
	 & Statistics	  &$\widehat{\alpha}$&$\widehat{d}$&$\widehat{D}$&&&&$\widehat{\alpha}$&$\widehat{d}$&$\widehat{D}$&\\
	\hline
	 &Mean &$1.615$&$0.132$&$0.170$&&  &&$1.594$&$0.130$&$0.155$&\\
     &RMSE &$0.015$&$0.017$&$0.029$&&  &&$0.006$&$0.019$&$0.044$&\\
     &MAE  &$0.015$&$0.017$&$0.029$&&  &&$0.005$&$0.019$&$0.044$&
	
\end{tabular}

\begin{tabular}{p{0.5cm} p{2.5cm} l  l l p{0.3cm} l l l l l p{0.5cm} }
	 &&
	 \multicolumn{2}{l}{$m=3$}&&&&&\multicolumn{2}{l}{$m=4$}&&\\
	 \cline{3-5}\cline{9-11}
	     &&$\widehat{\alpha}$&$\widehat{d}$&$\widehat{D}$&&&&$\widehat{\alpha}$&$\widehat{d}$&$\widehat{D}$&\\
	\hline
	  &	Mean &$1.573$&$0.128$&$0.136$&&  &&$1.596$&$0.128$&$0.145$&\\
        &  RMSE &$0.026$&$0.021$&$0.063$&&  &&$0.004$&$0.021$&$0.054$&\\
  & MAE  &$0.026 $&$0.021$&$0.063$&&  &&$0.004$&$0.021$&$0.054$&
	
\end{tabular}
\begin{tabular}{p{0.5cm} p{2.5cm} l  l l p{0.3cm} l l l l l p{0.5cm} }
	 &&
	 \multicolumn{2}{l}{$m=5$}&&&&&\multicolumn{2}{l}{$m=6$}&&\\
	 \cline{3-5}\cline{9-11}
	     &&$\widehat{\alpha}$&$\widehat{d}$&$\widehat{D}$&&&&$\widehat{\alpha}$&$\widehat{d}$&$\widehat{D}$&\\
	\hline

	  & Mean &$ 1.591$&$ 0.131$&$0.144$&&  &&$1.593$&$0.129$&$0.144$&\\
    & RMSE &$0.009$&$0.018$&$ 0.055$&&  &&$0.007$&$0.020$&$0.055$&\\
    & MAE  &$0.008$&$0.018 $&$0.055$&&  &&$0.006$&$0.020$&$0.055$&

\end{tabular}
\begin{tabular}{p{0.5cm} p{2.5cm} l  l l p{0.3cm} l l l l l p{0.5cm}}
	 &&
	 \multicolumn{2}{l}{$m=7$}&&&&&\multicolumn{2}{l}{$m=8$}&&\\
	 \cline{3-5}\cline{9-11}
	     &&$\widehat{\alpha}$&$\widehat{d}$&$\widehat{D}$&&&&$\widehat{\alpha}$&$\widehat{d}$&$\widehat{D}$&\\
	\hline
	 & Mean &$1.596$&$0.127$&$0.145$&&  &&$1.601$&$0.125$&$0.148$&\\
    & RMSE &$0.005$&$0.022$&$0.054$&&  &&$0.006$&$0.024$&$0.051$&\\
  &  MAE  &$0.004$&$0.022$&$0.054$&&  &&$0.004$&$0.024$&$0.051$&\\

     \hline
\end{tabular}\\
\caption{ECF estimation for model $1$, for $m=1$ to $m=8$.}
\label{tabmodel1ecf}
\end{center}
\end{table}

\begin{table}[H]
\begin{center}
\begin{tabular}{p{0.2cm} p{2.5cm} l l l l l l l l l l p{0.2cm}}
	\hline 
	 &&
	 \multicolumn{2}{l}{$m=1$}&&&&&\multicolumn{2}{l}{$m=2$}&&&\\
	 \cline{3-6}\cline{9-12}
	 &Statistics	  &$\widehat{\alpha}$&$\widehat{d}$&$\widehat{D}$&$\widehat{\phi}$&&&$\widehat{\alpha}$&$\widehat{d}$&$\widehat{D}$&$\widehat{\phi}$&\\
	\hline
	  &Mean &$1.726$&$0.217$&$0.233$&$0.464$&& &$1.683$&$0.125$&$0.252$&$0.633$&\\
        &RMSE &$0.288$&$0.208$&$0.263$&$0.366$&& &$0.171$&$0.117$&$0.132$&$0.192$&\\
     &MAE  &$0.168$&$0.174$&$0.147$&$0.292$&& &$0.122$&$0.100$&$0.089$ &$0.162$& 
\end{tabular}

\begin{tabular}{p{0.2cm} p{2.5cm} l l l l l l l l l l p{0.2cm}}
	
	 &&
	 \multicolumn{2}{l}{$m=3$}&&&&&\multicolumn{2}{l}{$m=4$}&&&\\
	 \cline{3-6}\cline{9-12}
	 	  &&$\widehat{\alpha}$&$\widehat{d}$&$\widehat{D}$&$\widehat{\phi}$&&&$\widehat{\alpha}$&$\widehat{d}$&$\widehat{D}$&$\widehat{\phi}$&\\
	\hline
	
	 &Mean &$1.663$&$0.125$&$0.243$&$0.641$& &&$1.650$&$0.128$&$0.235$&$0.644$&\\
         &RMSE &$0.135$&$0.109$&$0.101$&$0.167$& &&$0.120$&$0.106$&0.087&$0.154$&\\
    &MAE  &$0.100$&$0.095$&$0.073$&$0.140$& &&$0.085$&$0.093$&$0.063$&$0.130$&
	
\end{tabular}

\begin{tabular}{p{0.2cm} p{2.5cm} l l l l l l l l l l p{0.2cm}}
		 &&
	 \multicolumn{2}{l}{$m=5$}&&&&&\multicolumn{2}{l}{$m=6$}&&&\\
	 \cline{3-6}\cline{9-12}
	 	 &&$\widehat{\alpha}$&$\widehat{d}$&$\widehat{D}$&$\widehat{\phi}$&&&$\widehat{\alpha}$&$\widehat{d}$&$\widehat{D}$&$\widehat{\phi}$&\\
	\hline
	
	 &Mean &$1.634$&$0.160$&$0.228$&$0.606$&& &$1.624$&$0.169$&$0.222$&$0.591$&\\
     &RMSE &$0.098$&$0.097$&$0.070$&$0.119$&& &$0.082$&$0.099$&$0.057$&$0.103$&\\
     &MAE  &$0.068$&$0.078$&$0.047$&$0.089$&& &$0.061$&$0.079$&$0.040$&$0.083$&
	\end{tabular}
	
\begin{tabular}{p{0.2cm} p{2.5cm} l l l l l l l l l l p{0.2cm}}
		 &&
	 \multicolumn{2}{l}{$m=7$}&&&&&\multicolumn{2}{l}{$m=8$}&&&\\
	 \cline{3-6}\cline{9-12}
	 	 &&$\widehat{\alpha}$&$\widehat{d}$&$\widehat{D}$&$\widehat{\phi}$&&&$\widehat{\alpha}$&$\widehat{d}$&$\widehat{D}$&$\widehat{\phi}$&\\
	\hline
		 &Mean &$1.571$&$0.250$&$0.187$&$0.394$&& &$1.568$&$0.248$&$0.189$&$0.395$&\\
         &RMSE &$0.092$&$0.135$&$0.059$&$0.258$&& &$0.093$&$0.132$&$0.056$&$0.247$&\\
     &MAE  &$0.065$&$0.114$&$0.043$&$0.219$& &&$0.067$&$0.110$&$0.041$&$0.210$&\\
   
	\hline
\end{tabular}
\caption{ECF estimation for model $2$, for $m=1$ to $m=8$.}
\label{tabmodel2ecf}
\end{center}
\end{table}
\begin{table}[H]
\begin{center}
\begin{tabular}{p{0.2cm} p{2.5cm} l l l l l l l l l l p{0.2cm}}
	\hline 
	 &&
	 \multicolumn{2}{l}{$m=1$}&&&&&\multicolumn{2}{l}{$m=2$}&&&\\
	 \cline{3-6}\cline{9-12}
	 &Statistics	  &$\widehat{\alpha}$&$\widehat{d}$&$\widehat{D}$&$\widehat{\theta}$&&&$\widehat{\alpha}$&$\widehat{d}$&$\widehat{D}$&$\widehat{\theta}$&\\
	\hline
		 &Mean &$1.694$&$0.177$&$0.250$&$0.339$&& &$1.672$&$0.165$&$0.240$&$0.399$&\\
         &RMSE &$0.209$&$0.064$&$0.130$&$0.153$& &&$0.156$&$0.056$&$0.094$&$0.059$&\\
     &MAE  &$0.133$&$0.046$&$0.077$&$0.090$& &&$0.112$&$ 0.041$&$0.067$&$0.041$&
	
\end{tabular}

\begin{tabular}{p{0.2cm} p{2.5cm} l l l l l l l l l l p{0.2cm}}
		 &&
	 \multicolumn{2}{l}{$m=3$}&&&&&\multicolumn{2}{l}{$m=4$}&&&\\
	 \cline{3-6}\cline{9-12}
	 	  &&$\widehat{\alpha}$&$\widehat{d}$&$\widehat{D}$&$\widehat{\theta}$&&&$\widehat{\alpha}$&$\widehat{d}$&$\widehat{D}$&$\widehat{\theta}$&\\
	\hline
		 &Mean &$1.659$&$0.169$&$0.239$&$0.394$& &&$1.637$&$0.170$&$0.230$&$0.398$&\\
         &RMSE &$0.135$&$0.057$&$0.084$&$0.052$& &&$0.102$&$ 0.059$&$0.070$&$0.054$&\\
     &MAE  &$0.097$&$0.042$&$0.059$&$0.039$& &&$0.074$&$0.044$&$0.047$&$0.041$&
	
\end{tabular}

\begin{tabular}{p{0.2cm} p{2.5cm} l l l l l l l l l l p{0.2cm}}
		 &&
	 \multicolumn{2}{l}{$m=5$}&&&&&\multicolumn{2}{l}{$m=6$}&&&\\
	 \cline{3-6}\cline{9-12}
	 	 &&$\widehat{\alpha}$&$\widehat{d}$&$\widehat{D}$&$\widehat{\theta}$&&&$\widehat{\alpha}$&$\widehat{d}$&$\widehat{D}$&$\widehat{\theta}$&\\
	\hline
		 &Mean &$1.630$&$0.170$&$0.225$&$0.398$&& &$1.534$&$0.126$&$0.143$&$0.343$&\\
        &RMSE &$0.095$&$0.064$&$0.066$&$0.059$& &&$0.066$&$0.023$&$0.056$&$0.057$&\\
    &MAE  &$0.067$&$0.046$&$0.043$&$0.044$& &&$0.065$&$0.023$&$0.056$&$0.056$&
	\end{tabular}

\begin{tabular}{p{0.2cm} p{2.5cm} l l l l l l l l l l p{0.2cm}}
		 &&
	 \multicolumn{2}{l}{$m=7$}&&&&&\multicolumn{2}{l}{$m=8$}&&&\\
	 \cline{3-6}\cline{9-12}
	 	 &&$\widehat{\alpha}$&$\widehat{d}$&$\widehat{D}$&$\widehat{\theta}$&&&$\widehat{\alpha}$&$\widehat{d}$&$\widehat{D}$&$\widehat{\theta}$&\\
	\hline
	 &Mean &$1.622$&$0.171$&$0.223$&$0.396$&& &$1.619$&$0.170$&$0.222$&$0.396$&\\
    &RMSE &$0.090$&$0.071$&$0.070$&$0.067$& &&$0.083$&$0.069$&$0.075$&$0.068$&\\
     &MAE  &$0.063$&$0.049$&$0.043$&$0.051$& &&$0.061$&$0.050$&$0.044$&$0.053$&\\
    \hline
\end{tabular}
\caption{ECF estimation for model $3$, for $m=1$ to $m=8$.}
\label{tabmodel3ecf}
\end{center}
\end{table}


\begin{table}[H]
\begin{center}
\begin{tabular}{p{1.5cm} l l l l l l l l l l l }
	\hline
	 &
	 \multicolumn{2}{l}{$m=1$}&&&&&\multicolumn{2}{l}{$m=2$}&&&\\
	 \cline{2-6}\cline{8-12}
Statistics&$\widehat{\alpha}$&$\widehat{d}$&$\widehat{D}$&$\widehat{\phi}$&$\widehat{\theta}$&&$\widehat{\alpha}$&$\widehat{d}$&$\widehat{D}$&$\widehat{\phi}$&$\widehat{\theta}$\\
\hline
     Mean &$1.488$&$0.356$&$0.0008$&$0.002$&$0.331$& &$1.508$&$0.275$&$0.026$&$0.276$&$0.411$\\
     RMSE &$0.111$&$0.206$&$0.199$&$0.597$&$0.069$& &$0.091$&$0.125$&$0.173$&$0.323$&$0.012$\\
   MAE  &$0.111$&$0.206$&$0.199$&$0.597$&$0.068$& &$0.091$&$0.125$&$0.173$&$0.323$&$0.011$
\end{tabular}

\begin{tabular}{p{1.5cm} l l l l l l l l l l l }
	 &
	 \multicolumn{2}{l}{$m=3$}&&&&&\multicolumn{2}{l}{$m=4$}&&&\\
	 \cline{2-6}\cline{8-12}
Statistics&$\widehat{\alpha}$&$\widehat{d}$&$\widehat{D}$&$\widehat{\phi}$&$\widehat{\theta}$&&$\widehat{\alpha}$&$\widehat{d}$&$\widehat{D}$&$\widehat{\phi}$&$\widehat{\theta}$\\
\hline
     Mean &$1.513$&$0.262$&$0.045$&$0.288$&$0.402$& &$1.514$&$0.222$&$0.089$&$0.349$&$0.365$\\
     RMSE &$0.086$&$0.112$&$0.155$&$0.312$&$0.005$& &$0.085$&$0.073$&$0.110$&$0.250$&$0.034$\\
     MAE  &$0.086$&$0.112$&$0.154$&$0.311$&$0.004$& &$0.085$&$0.072$&$0.110$&$0.250$&$0.034$
\end{tabular}

\begin{tabular}{p{1.5cm} l l l l l l l l l l l}
	 &
	 \multicolumn{2}{l}{$m=5$}&&&&&\multicolumn{2}{l}{$m=6$}&&&\\
	 \cline{2-6}\cline{8-12}
Statistics&$\widehat{\alpha}$&$\widehat{d}$&$\widehat{D}$&$\widehat{\phi}$&$\widehat{\theta}$&&$\widehat{\alpha}$&$\widehat{d}$&$\widehat{D}$&$\widehat{\phi}$&$\widehat{\theta}$\\
\hline
     Mean &$1.520$&$0.178$&$0.118$&$0.451$&$0.351$& &$1.513$&$0.147$&$0.133$&$0.510$&$0.334$\\
     RMSE &$0.080$&$0.031$&$0.081$&$0.151$&$0.049$& &$0.087$&$0.017$&$0.066$&$0.095$&$0.066$\\
     MAE  &$0.079$&$0.028$&$0.081$&$0.148$&$0.048$& &$0.086$&$0.013$&$0.066$&$0.089$&$0.065$
\end{tabular}

\begin{tabular}{p{1.5cm} l l l l l l l l l l l}
	
	 &
	 \multicolumn{2}{l}{$m=7$}&&&&&\multicolumn{2}{l}{$m=8$}&&&\\
	 \cline{2-6}\cline{8-12}
Statistics&$\widehat{\alpha}$&$\widehat{d}$&$\widehat{D}$&$\widehat{\phi}$&$\widehat{\theta}$&&$\widehat{\alpha}$&$\widehat{d}$&$\widehat{D}$&$\widehat{\phi}$&$\widehat{\theta}$\\
\hline
     Mean &$1.492$&$0.135$&$0.139$&$0.525$&$0.305$& &$1.456$&$0.132$&$0.142$&$0.503$&$0.268$\\
    RMSE &$0.109$&$0.028$&$0.062$&$0.089$&$0.096$& &$0.148$&$0.041$&$0.060$&$0.124$&$0.135$\\
     MAE  &$0.107$&$0.022$&$0.060$&$0.077$&$0.094$& &$0.143$&$0.032$&$0.057$&$0.105$&$0.131$\\
  
 \hline
\end{tabular}
\caption{ECF estimation for model $4$, for $m=1$ to $m=8$.}
\label{tabmodel4ecf}
\end{center}
\end{table}
\subsection{Monte Carlo study comparing ECF procedure and Two-Step Method}\noindent
In this section, we compare the finite sample performance of the ECF method to the Two-Step Method discussed in Section \ref{secMCMCW}. The simulation results are summarized in Tables \ref{tabcompmod1}, \ref{tabcompmod2}, \ref{tabcompmod3} and \ref{tabcompmod4}.
\begin{itemize}
\item When we consider Table \ref{tabcompmod1} displaying the estimation results for ARFISMA$(0,\ d,\ 0)\times(0,\ D,\ 0)_{s}$-S$\alpha$S model, we see that all methods perform well, as the RMSE and MAE are in most cases small. The estimation parameters from the TSM are better than those obtained by the ECF approach.
\item Tables \ref{tabcompmod2}, \ref{tabcompmod3} and \ref{tabcompmod4} display the estimation results, when there are long memory and short memory components simultaneously. From these tables, we observe that the estimation results are acceptable for all methods, in most cases. However, it can be remarked that the estimation of the parameters is disturbed. Indeed, the RMSE and MAE obtained from these tables are larger than those presented in Table \ref{tabcompmod1}. Hence, this shows that the impact of the short memory parameters in the estimates. Comparing the two methods, it can be seen that better estimates are obtained, in general, from the ECF procedure than the TSM. Moreover, an advantage of the ECF method is that, as a consistent procedure it estimates all of the parameters simultaneously.   
 \end{itemize}
\begin{table}[H]
\begin{center}
\begin{tabular}{p{0.3cm} p{1.7cm} p{1.3cm} p{1.3cm} c c c c p{1.3cm} p{1.3cm} p{1.3cm} }
	\hline
	 &&
	 \multicolumn{2}{l}{ECF (m=1)}&&&&&&\multicolumn{2}{l}{TSM}\\
	 \cline{3-5}\cline{9-11}
	 	&Statistics&$\hat{\alpha}$&$\hat{d}$&$\hat{D}$&&&&$\hat{\alpha}$&$\hat{d}$&$\hat{D}$\\
	\hline
	 &Mean &$1.6158$&$0.1321$&$0.1707$&& &&$1.6105$&$0.1565$&$0.1844$\\
     &RMSE &$0.0158$&$0.0178$&$0.0292$&& &&$0.0108$&$0.0132$&$0.0193$\\
    &MAE  &$0.0158$&$0.0178$&$0.0292$&& &&$0.0105$&$0.0106$&$0.0164$\\
     \hline
\end{tabular}
\caption{Monte Carlo study to compare ECF procedure with TSM for model $1$.}
\label{tabcompmod1}
\end{center}
\end{table}

\begin{table}[H]
\begin{center}
\begin{tabular}{p{0.3cm} p{1.5cm} p{0.8cm} p{0.8cm} p{0.8cm} c  c c p{0.8cm} p{0.8cm} p{0.8cm} p{0.8cm} p{0.1cm}}
	\hline
	 &&
	 \multicolumn{2}{l}{ECF (m=6)}&&&&&&\multicolumn{2}{l}{TSM}&&\\
	 \cline{3-6}\cline{9-12}
	 	 &Statistics&$\hat{\alpha}$&$\hat{d}$&$\hat{D}$&$\hat{\phi}$&&&$\hat{\alpha}$&$\hat{d}$&$\hat{D}$&$\hat{\phi}$&\\
	\hline
		 &Mean &$1.624$&$0.169$&$0.222$&$0.591$&& &$1.603$&$0.164$&$0.200$&$0.587$&\\
       &RMSE &$0.082$&$0.099$&$0.057$&$0.103$&& &$0.008$&$0.164$&$0.079$&$0.104$&\\
    &MAE  &$0.061$&$0.079$&$0.040$&$0.083$&& &$0.006$&$0.077$&$0.026$&$0.073$&\\
        \hline
	\end{tabular}
\caption{Monte Carlo study to compare ECF procedure with TSM for model $2$.}
	\label{tabcompmod2}
\end{center}
\end{table}
\begin{table}[H]
\begin{center}
\begin{tabular}{p{0.3cm} p{1.5cm} p{0.8cm} p{0.8cm} p{0.8cm} c  c c  p{0.8cm} p{0.8cm} p{0.8cm} p{0.8cm} p{0.1cm}}
	\hline
	 &&
	 \multicolumn{2}{l}{ECF (m=6)}&&&&&&\multicolumn{2}{l}{TSM}&&\\
	 \cline{3-6}\cline{9-12}
	 	 &Statistics&$\hat{\alpha}$&$\hat{d}$&$\hat{D}$&$\hat{\theta}$&&&$\hat{\alpha}$&$\hat{d}$&$\hat{D}$&$\hat{\theta}$&\\
	\hline

	 &Mean &$1.534$&$0.126$&$0.143$&$0.343$&& &$1.256$&$0.157$&$0.204$&$0.395$&\\
    &RMSE &$0.066$&$0.023$&$0.056$&$0.057$&& &$0.344$&$0.064$&$0.068$&$0.048$&\\
    &MAE  &$0.065$&$0.023$&$0.056$&$0.056$&& &$0.343$&$0.026$&$0.022$&$0.029$&\\
	\hline
\end{tabular}
\caption{Monte Carlo study to compare ECF procedure with TSM for model $3$.}
\label{tabcompmod3}
\end{center}
\end{table}

\begin{table}[H]
\begin{center}
\begin{tabular}{p{1.5cm} c  c  c c c c c c c c c }
	\hline
	 &&
	 \multicolumn{2}{l}{ECF (m=6)}&&&&&&\multicolumn{2}{l}{TSM}&\\
	 \cline{2-6}\cline{8-12}
Statistics&$\hat{\alpha}$&$\hat{d}$&$\hat{D}$&$\hat{\phi}$&$\hat{\theta}$&&$\hat{\alpha}$&$\hat{d}$&$\hat{D}$&$\hat{\phi}$&$\hat{\theta}$\\
\hline
  Mean &$1.513$&$0.147$&$0.133$&$0.510$&$0.334$& &$1.286$&$0.146$&$0.182$&$0.610$&$0.172$\\
 RMSE &$0.087$&$0.017$&$0.066$&$0.095$&$0.066$& &$0.316$&$0.050$&$0.025$&$0.062$&$0.233$\\
 MAE  &$0.086$&$0.013$&$0.066$&$0.089$&$0.065$& &$0.313$&$0.040$&$0.021$&$0.049$&$0.227$\\
  \hline
\end{tabular}
\caption{Monte Carlo study to compare ECF with TSM for model $4$.}
\label{tabcompmod4}
\end{center}
\end{table}
\section{Conclusion}\label{Sec.6}\noindent
In this article, we studied the ECF method developed by Knight and Yu \cite{KY}. The method is illustrated to estimate simultaneously the parameters of stable ARFISMA models introduced by Diongue {\it et al.} \cite{DDM}. Under some conditions, the resulting estimators are shown to be consistent and asymptotically normally distributed. For comparison purpose, we have also consider a Two-Step Method (TSM) composed by the MCMC Whittle procedure developed in Ndongo {\it et al.} \cite{NDDS} and the MLE approach introduced in Alvarez and Olivares \cite{AO2005}. Finite sample behaviours of these methods were studied through Monte Carlo simulations. It is found in general, that the ECF method is better than the TSM, particularly, when short memory components are present in the model. The simulation results show also the impact of the size of moving blocks and the short memory parameters in the estimate.\\

\noindent It is very important to remark that there are two difficulties with the implementation of the characteristic function-based estimators (the use of optimal weighting function and large set of moment conditions) as noted in Carrasco {\it et al.} (\cite{CCFG2007}). They solved the two problems in the framework of the GMM with continuum of moment conditions. However, the models considered here is non-Markovian, and then the conditional Characteristic Function (CCF) is unknown and difficult to estimate. Hence, the GMM based on the CCF is not feasible. On the other hand, the Joint Characteristic Function (JCF) of stable ARFISMA models is known but the GMM-JCF will not be efficient, since the process is not Markovian (see Carrasco {\it et al.} (\cite{CCFG2007}) for more details). Thus, it will be interesting to compare the sub-optimal ECF method with an exponential weighting function and the GMM based on the JCF, through Monto Carlo simulations. This problem will be examined in a forthcoming paper.

\section*{Appendix A: proof of Theorem \ref{Th3.1}}
$$c\left(r_1,\ldots,r_{m+1};\psi\right)=\mathbb{E}\left[\exp\left\{ir_{1}X_{t-m}+ir_{2}X_{t-m+1}+\cdots+ir_{m+1}X_{t}\right\}\right]$$
 Using equation (\ref{MArepres}), we can write
 $$c\left(r_1,\ldots,r_{m+1};\psi\right)=\mathbb{E}\left[\exp\left\{ir_{1}\sum_{j=0}^{\infty}c_jZ_{t-m-j}+ir_{2}\sum_{j=0}^{\infty}c_jZ_{t-m+1-j}+\ldots+ir_{m+1}\sum_{j=0}^{\infty}c_jZ_{t-j}\right\}\right]$$
 Developing and regrouping the terms in $Z_{t-m-j}$ for $j=0$ to $\infty$, we obtain
 \begin{eqnarray} c\left(r;\psi\right)&=&\mathbb{E}\left[\exp\left\{i\sum_{j=0}^{\infty}(\sum_{l=0}^{m}r_{l+1}c_{j+l})Z_{t-m-j}+i\sum_{h=0}^{m-1}r_{h+2}c_{h}Z_{t-m+1}+\cdots+ir_{m+1}c_0Z_{t}\right\}\right]\nonumber\\ &=&\mathbb{E}\left[\exp\left\{i\sum_{j=0}^{\infty}(\sum_{l=0}^{m}r_{l+1}c_{j+l})Z_{t-m-j}+i\sum_{l=2}^{m+1}(\sum_{h=0}^{m+1-l}r_{h+l}c_{h})Z_{t-m-1+l}\right\}\right]\nonumber
 \end{eqnarray}
 Using the independence of the $Z_t$, we have
 \begin{eqnarray} c\left(r;\psi\right)&=&\mathbb{E}\left[\exp\left\{i\sum_{j=0}^{\infty}(\sum_{l=0}^{m}r_{l+1}c_{j+l})Z_{t-m-j}\right\}\right]\times\mathbb{E}\left[\exp\left\{i\sum_{l=2}^{m+1}(\sum_{h=0}^{m+1-l}r_{h+l}c_{h})Z_{t-m-1+l}\right\}\right]\nonumber
 \end{eqnarray}
 Let $$A=\mathbb{E}\left[\exp\left\{i\sum_{j=0}^{\infty}(\sum_{l=0}^{m}r_{l+1}c_{j+l})Z_{t-m-j}\right\}\right]$$  $$B=\mathbb{E}\left[\exp\left\{i\sum_{l=2}^{m+1}(\sum_{h=0}^{m+1-l}r_{h+l}c_{h})Z_{t-m-1+l}\right\}\right]$$
Calculate $A$ and $B$
\begin{eqnarray}
B&=&\mathbb{E}\left[\exp\left\{i\sum_{l=2}^{m+1}(\sum_{h=0}^{m+1-l}r_{h+l}c_{h})Z_{t-m-1+l}\right\}\right]\nonumber\\
&=&\mathbb{E}\left[\prod_{l=2}^{m+1}\exp\left\{i(\sum_{h=0}^{m+1-l}r_{h+l}c_{h})Z_{t-m-1+l}\right\}\right]\nonumber\\
&=&\prod_{l=2}^{m+1}\mathbb{E}\left[\exp\left\{i(\sum_{h=0}^{m+1-l}r_{h+l}c_{h})Z_{t-m-1+l}\right\}\right],\ \textrm{for indenpence of the}\ Z_t\nonumber\\
&=&\prod_{l=2}^{m+1} \Phi_{Z}\left(\sum_{h=0}^{m+1-l}r_{h+l}c_{h}\right)\nonumber\\
&=&\prod_{l=2}^{m+1}\exp\left\{-\left|\sum_{h=0}^{m+1-l}r_{h+l}c_{h}\right|^{\alpha}\right\}\nonumber\\
&=&\exp\left\{-\sum_{l=2}^{m+1}\left|\sum_{h=0}^{m+1-l}r_{h+l}c_{h}\right|^{\alpha}\right\}\nonumber
\end{eqnarray}
For calculation of $A$, let $Y=\displaystyle\sum_{j=0}^{\infty}(\sum_{l=0}^{m}r_{l+1}c_{j+l})Z_{t-m-j}$ and remark that $A=\Phi_{Y}(1)$, where $\Phi_Y$ is the CF of $Y$. \\The sequence $\left\{Z_t\right\}$ is i.i.d symmetric $\alpha$-stable and $\displaystyle\sum_{j=0}^{\infty}\left|\sum_{l=0}^{m}r_{l+1}c_{j+l}\right|^{\alpha}<\infty$, then according to Cline (\cite{Cli}) or Cline and Brockwell (\cite{CliBroc}), we have Y is also symmetric $\alpha$-stable and
$$Y\stackrel{d}{=}\left(\sum_{j=0}^{\infty}\left|\sum_{l=0}^{m}r_{l+1}c_{j+l}\right|^{\alpha}\right)^{1/\alpha}Z_1.$$
As consequently
\begin{eqnarray}
\Phi_{Y}\left(u\right)&=&\mathbb{E}\left[\exp\left\{i u\sum_{j=0}^{\infty}(\sum_{l=0}^{m}r_{l+1}c_{j+l})Z_{t-m-j}\right\}\right]\nonumber\\
&=&\mathbb{E}\left[\exp\left\{i u\left(\sum_{j=0}^{\infty}\left|\sum_{l=0}^{m}r_{l+1}c_{j+l}\right|^{\alpha}\right)^{1/\alpha}Z_{1}\right\}\right]\nonumber\\
&=&\Phi_{Z_1}\left(u\left(\sum_{j=0}^{\infty}\left|\sum_{l=0}^{m}r_{l+1}c_{j+l}\right|^{\alpha}\right)^{1/\alpha}\right)\nonumber\\
&=&\exp\left\{-\left|u\right|^{\alpha}\sum_{j=0}^{\infty}\left|\sum_{l=0}^{m}r_{l+1}c_{j+l}\right|^{\alpha}\right\}\nonumber
\end{eqnarray}
Hence $A=\Phi_{Y}\left(1\right)=\exp\left\{-\displaystyle\sum_{j=0}^{\infty}\left|\sum_{l=0}^{m}r_{l+1}c_{j+l}\right|^{\alpha}\right\}.$\\
Finally
\begin{eqnarray}
c\left(r;\psi\right)&=&A\times B\nonumber\\
&=&\exp\left\{-\sum_{j=0}^{\infty}\left|\sum_{l=0}^{m}r_{l+1}c_{j+l}\right|^{\alpha}\right\}\times\exp\left\{-\sum_{l=2}^{m+1}\left|\sum_{h=0}^{m+1-l}r_{h+l}c_{h}\right|^{\alpha}\right\}\nonumber\\
&=&\exp\left\{-\sum_{j=0}^{\infty}\left|\sum_{l=0}^{m}r_{l+1}c_{j+l}\right|^{\alpha}-\sum_{l=2}^{m+1}\left|\sum_{h=0}^{m+1-l}r_{h+l}c_{h}\right|^{\alpha}\right\}\nonumber.
\end{eqnarray}
\hspace{\stretch{1}}\rule{1ex}{1ex}
\section*{Appendix B: proof of Theorem \ref{Th3.9}}
In this Section, we present the proof of Theorem \ref{Th3.9} which consists, under Assumptions \ref{Assump3.3} and \ref{Assump3.4}, to verify the hypothesis of Theorem 2.1 of Knight and Yu \cite{KY}, namely Assumptions ($A_2$), ($A_3$), ($A_5$)  and ($A_6$).
\begin{enumerate}
\item[$(A_2$):] With probability one, $I_n(\psi)$ is twice continuously differentiable under the integral sign with respect to $\psi$ over $\Psi$.\\\\
Let $f(r;\psi)=\left|c_{n}(r)-c(r;\psi)\right|^{2}g(r)$, then we have:
\begin{eqnarray}
\left|\frac{\partial}{\partial\psi}f\left(r;\psi\right)\right|&=&2~g(r)\left|c_{n}(r)-c(r;\psi)\right|\left|\frac{\partial~c(r;\psi)}{\partial\psi}\right|\nonumber\\
                  &\leq&2~g(r)\left(\left|c_{n}(r)\right|+\left|c(r;\psi)\right|\right)\left|\frac{\partial~c(r;\psi)}{\partial\psi}\right|\nonumber\\
                  &\leq& 4~g(r)\left|\frac{\partial~c(r;\psi)}{\partial\psi}\right|.\nonumber
\end{eqnarray}
Defining $$U(r;\psi)=\sum_{j=0}^{\infty}\left|\sum_{l=0}^{m}r_{l+1}c_{j+l}\right|^{\alpha}+\sum_{l=2}^{m+1}\left|\sum_{h=0}^{m+1-l}r_{h+l}c_{h}\right|^{\alpha},$$
we can rewrite the previous inequality as:
\begin{eqnarray}
\left|\frac{\partial}{\partial\psi}f\left(r;\psi\right)\right|&\leq& 4~g(r)\exp(-U(r;\psi))\left|\frac{\partial~U(r;\psi)}{\partial\psi}\right|\nonumber\\
                                           &\leq& 4~g(r)\left|\frac{\partial~U(r;\psi)}{\partial\psi}\right|.\nonumber
\end{eqnarray}
$\bullet$ {\bf Let us show that $I_n(\psi)$ is of class $\mathcal{C}^1$ with respect to $\delta=(d,\ D,\ \phi ,\ \theta,\ \Phi,\ \Theta)$},
where $\phi = (\phi_{1},\ldots,\phi_{p})$, $\theta = (\theta_{1},\ldots,\theta_{q})$, $\Phi=(\Phi_{1},\ldots,\Phi_{P})$ and $\Theta = (\Theta_{1},\ldots,\Theta_{Q})$.\\
Let $S(r;\psi)=\displaystyle\sum_{j=0}^{\infty} u_j(r;\psi)$, with $u_j(r;\psi)=\left|\displaystyle\sum_{l=0}^{m}r_{l+1}c_{j+l}\right|^{\alpha}$. Then we get:
\begin{eqnarray}
\left|\frac{\partial~u_j(r;\psi)}{\partial\delta}\right|&=&\alpha\left|\sum_{l=0}^{m}r_{l+1}c_{j+l}\right|^{\alpha-1}\left|\sum_{l=0}^{m}r_{l+1}\frac{\partial~c_{j+l}}{\partial\delta}\right|\nonumber\\
&\leq& 2\left(\sum_{l=0}^{m}|r_{l+1}c_{j+l}|\right)^{\alpha-1}\left(\sum_{l=0}^{m}|r_{l+1}a_{j+l}|\right),\ \textrm{with}\ a_j=\frac{\partial~c_{j}}{\partial\delta}.\nonumber
\end{eqnarray}
Applying H$\ddot{o}$lder's inequality to the previous relationship, we obtain:
\begin{eqnarray}
\left|\frac{\partial~u_j(r;\psi)}{\partial\delta}\right|&\leq& 2\left(\sum_{l=0}^{m}|r_{l+1}|^{\frac{\alpha}{\alpha-1}}\right)^{\alpha-1}\left(\sum_{l=0}^{m}|c_{j+l}|^{\alpha}\right)^{1-\frac{1}{\alpha}}\left(\sum_{l=0}^{m}|a_{j+l}|^{\alpha}\right)^{\frac{1}{\alpha}},\nonumber\\
&\leq& 2\left(\sum_{l=0}^{m}|r_{l+1}|\right)^{\alpha}\left(\sum_{l=0}^{m}|c_{j+l}|^{\alpha}\right)^{1-\frac{1}{\alpha}}\left(\sum_{l=0}^{m}|a_{j+l}|^{\alpha}\right)^{\frac{1}{\alpha}}.\nonumber
\end{eqnarray}
- There exists $l_0$ such that $|c_{j+l}|^{\alpha} \leq |c_{j+l_0}|^{\alpha}$, for all $l=0,\ldots,m$.
\begin{equation}\label{eql0}
  \Rightarrow \left(\sum_{l=0}^{m}|c_{j+l}|^{\alpha}\right)^{1-\frac{1}{\alpha}} \leq (m+1)^{1-\frac{1}{\alpha}}|c_{j+l_0}|^{\alpha-1}.
\end{equation}
- There exists $l_1$ such that $|a_{j+l}|^{\alpha} \leq |a_{j+l_1}|^{\alpha}$, for all $l=0,\ldots,m$.
\begin{equation}\label{eql1}
  \Rightarrow \left(\sum_{l=0}^{m}|a_{j+l}|^{\alpha}\right)^{\frac{1}{\alpha}} \leq (m+1)^{\frac{1}{\alpha}}|a_{j+l_1}|.
 \end{equation}
  Using expression found in (\ref{eql0}) and (\ref{eql1}), we get:
  \begin{eqnarray}
  \left|\frac{\partial~u_j(r;\psi)}{\partial\delta}\right| &\leq &2(m+1)\left(\sum_{l=0}^{m}|r_{l+1}|\right)^{\alpha}|c_{j+l_0}|^{\alpha-1}|a_{j+l_1}|\nonumber\\
        &\leq &(m+1)\left(\sum_{l=0}^{m}|r_{l+1}|\right)^{\alpha} w_j,\quad\textrm{with}\quad w_j=|c_{j+l_0}|^{2(\alpha-1)} + |a_{j+l_1}|^2.\nonumber
  \end{eqnarray}
  According to Assumption \ref{Assump3.3}, the series $\sum w_j$ is convergent. Consequently, we get $U(r;\ \psi)$ is differentiable with respect to $\delta$ and we can write:
  \begin{eqnarray}
\left|\frac{\partial~U(r;\psi)}{\partial\delta}\right|&=&\frac{\partial~S(r;\psi)}{\partial\delta}+\alpha\sum_{l=2}^{m+1}\left|\sum_{h=0}^{m+1-l}r_{h+l}c_{h}\right|^{\alpha-1}\left|\sum_{h=0}^{m+1-l}r_{h+l}\frac{\partial~c_h}{\partial\delta}\right|\nonumber\\
&\leq&(m+1)\left(\sum_{l=0}^{m}|r_{l+1}|\right)^{\alpha}\sum_{j=0}^{\infty} w_j\nonumber\\
&+&\alpha\sum_{l=2}^{m+1}\left[\left(\sum_{h=0}^{m+1-l}|r_{h+l}c_{h}|\right)^{\alpha-1}\left(\sum_{h=0}^{m+1-l}|r_{h+l}a_h|\right)\right]\nonumber\\
 &\leq& M~(m+1)\left(\sum_{l=0}^{m}|r_{l+1}|\right)^{\alpha}+\alpha\sum_{l=2}^{m+1}\left[\left(\sum_{h=0}^{m+1-l}|r_{h+l}c_{h}|\right)^{\alpha-1}\left(\sum_{h=0}^{m+1-l}|r_{h+l}a_h|\right)\right].\nonumber
\end{eqnarray}
\begin{equation}\label{eqAA}
\textrm{Let}\quad A=\alpha\sum_{l=2}^{m+1}\left[\left(\sum_{h=0}^{m+1-l}|r_{h+l}c_{h}|\right)^{\alpha-1}\left(\sum_{h=0}^{m+1-l}|r_{h+l}a_h|\right)\right].~~~~~~~~~~~~~~~~~~~~~~~~~~~~~~~~~~~~~~
\end{equation}
Applying H$\ddot{o}$lder's inequality to $A$, we have:
\begin{eqnarray}
 A &\leq& 2\sum_{l=2}^{m+1}\left[\left(\sum_{h=0}^{m+1-l}|r_{h+l}|\right)^{\alpha-1}\left(\sum_{h=0}^{m+1-l}|r_{h+l}|^2\right)^{1/2}\left(\sum_{h=0}^{m+1-l}|a_{h}|^2\right)^{1/2}\right]\nonumber\\
&\leq & 2\sum_{l=2}^{m+1}\left[\left(\sum_{h=0}^{m+1-l}|r_{h+l}|\right)^{\alpha-1}\left(\sum_{h=0}^{m+1-l}|r_{h+l}|\right)\left(\sum_{h=0}^{m+1-l}|a_{h}|^2\right)^{1/2}\right]\nonumber\\
&\leq & 2\sum_{l=2}^{m+1}\left[\left(\sum_{h=0}^{m+1-l}|r_{h+l}|\right)^{\alpha}\left(\sum_{h=0}^{+\infty}|a_{h}|^2\right)^{1/2}\right]\nonumber\\
&\leq & 2~M'\sum_{l=2}^{m+1}\left(\sum_{h=0}^{m+1-l}|r_{h+l}|\right)^{\alpha},\quad\textrm{using Assumption \ref{Assump3.3}}.\nonumber
\end{eqnarray}
So we get:
\begin{eqnarray}\label{eq2}
\left|\frac{\partial~U(r;\psi)}{\partial\delta}\right| &\leq & M~(m+1)\left(\sum_{l=0}^{m}|r_{l+1}|\right)^{\alpha} + 2~M'\sum_{l=2}^{m+1}\left(\sum_{h=0}^{m+1-l}|r_{h+l}|\right)^{\alpha}\nonumber\\
&\leq & M~(m+1)H_1(r) + 2~M'\sum_{l=2}^{m+1} H_1(r;l),\nonumber
\end{eqnarray}
where
\begin{equation}
\left(\sum_{l=0}^{m}|r_{l+1}|\right)^{\alpha}\leq H_1(r)=\left\{\begin{array}{ll}
\left(\displaystyle\sum_{l=0}^{m}|r_{l+1}|\right)^{2},\qquad \textrm{si}\ \displaystyle\sum_{l=0}^{m}|r_{l+1}|>1\\\nonumber\\
\displaystyle\sum_{l=0}^{m}|r_{l+1}|,\qquad \textrm{si}\ \displaystyle\sum_{l=0}^{m}|r_{l+1}|\leq 1.\\\nonumber\\
\end{array}\right.
\end{equation}
and
\begin{equation}
\left(\sum_{h=0}^{m+1-l}|r_{h+l}|\right)^{\alpha}\leq H_1(r;l)=\left\{\begin{array}{ll}
\left(\displaystyle\sum_{h=0}^{m+1-l}|r_{h+l}|\right)^{2},\qquad \textrm{si}\ \displaystyle\sum_{h=0}^{m+1-l}|r_{h+l}|>1\\\nonumber\\
\displaystyle\sum_{h=0}^{m+1-l}|r_{h+l}|,\qquad \textrm{si}\ \displaystyle\sum_{h=0}^{m+1-l}|r_{h+l}|\leq 1.\\\nonumber\\
\end{array}\right.
\end{equation}
Letting $\Gamma_1(r)=M~(m+1)H_1(r) + 2~M'\displaystyle\sum_{l=2}^{m+1} H_1(r;l)$, we get:
$$\left|\frac{\partial}{\partial\delta}f\left(r;\psi\right)\right| \leq 4~\Gamma_1(r)~g(r).$$
Using Assumption \ref{Assump3.4}, it is easy to show that:
$$\int\ldots\int \Gamma_1(r)~g(r)\ dr<\infty.$$
Hence $I_n(\psi)$ is of class $\mathcal{C}^1$ with respect to $\delta$.\\\\
$\bullet$ {\bf Let us show that $I_n(\psi)$ is of class $\mathcal{C}^1$ with respect to $\alpha$}.\\
Using the same approach as before, we show that $U(r;\ \psi)$ is differentiable with respect to $\alpha$. Thus we get:
\begin{eqnarray}
\left|\frac{\partial~U(r;\psi)}{\partial\alpha}\right|&=& \left|\frac{\partial}{\partial\alpha}S(r;\ \psi) + \frac{\partial}{\partial\alpha}\sum_{l=2}^{m+1}\left|\sum_{h=0}^{m+1-l}r_{h+l}c_{h}\right|^{\alpha}\right|\nonumber\\
&=&\left|\frac{\partial}{\partial\alpha}S(r;\ \psi) + \sum_{l=2}^{m+1}\left|\sum_{h=0}^{m+1-l}r_{h+l}c_{h}\right|^{\alpha}\ln \left|\sum_{h=0}^{m+1-l}r_{h+l}c_{h}\right|\right|.\nonumber
\end{eqnarray}
Applying the H$\ddot{o}$lder's inequality and using the approach as before, we have:
\begin{eqnarray}
\left|\frac{\partial~U(r;\psi)}{\partial\alpha}\right|&\leq& (m+1)^2~H_2(r)\sum_{j=0}^{+\infty}\left|c_{j+l_0}\right|^{2} + \sum_{l=2}^{m+1}\left[ \left|\sum_{h=0}^{m+1-l}r_{l+h}c_{h}\right|^{\alpha + 1} \right]\nonumber\\
&\leq& M~(m+1)^2~H_2(r)+ \sum_{l=2}^{m+1} H_2(r;l):=\Gamma_2(r) \nonumber
\end{eqnarray}
where
\begin{equation}
H_2(r)=\left\{\begin{array}{ll}
\left(\displaystyle\sum_{l=0}^{m}|r_{l+1}|\right)^{3},\qquad \textrm{si}\ \displaystyle\sum_{l=0}^{m}|r_{l+1}|>1\\\nonumber\\
\left(\displaystyle\sum_{l=0}^{m}|r_{l+1}|\right)^2,\qquad \textrm{si}\ \displaystyle\sum_{l=0}^{m}|r_{l+1}|\leq 1,\\\nonumber\\                                                  
\end{array}\right.
\end{equation}
and
\begin{equation}
 H_2(r;l)=\left\{\begin{array}{ll}
\left(\displaystyle\sum_{h=0}^{m+1-l}|r_{h+l}|\right)^{3},\qquad \textrm{si}\ \displaystyle\sum_{h=0}^{m+1-l}|r_{h+l}|>1\\\nonumber\\
\left(\displaystyle\sum_{h=0}^{m+1-l}|r_{h+l}|\right)^2,\qquad \textrm{si}\ \displaystyle\sum_{h=0}^{m+1-l}|r_{h+l}|\leq 1.\\\nonumber~~~~~~~~~~~~~~~~~~~~~~~~~~~\\
\end{array}\right.
\end{equation}
Thus we obtain:
$$\left|\frac{\partial f\left(r;\psi\right)}{\partial\alpha}\right|\leq 4~\Gamma_2(r)~g(r),$$
and then using Assumption \ref{Assump3.4}, we have $\displaystyle\int\ldots\int \Gamma_2(r)~g(r)\ dr<\infty$, and hence $I_n(\psi)$ is of class $\mathcal{C}^1$ with respect to $\alpha$.\\\\ Now we will show that $I_n(\psi)$ is of class $\mathcal{C}^2$ with respect to $\psi$. Thus we get:
\begin{eqnarray}
\left|\frac{\partial^2}{\partial\psi\partial\psi'}f\left(r;\psi\right)\right|&=&2~g(r)\left|\frac{\partial~c(r;\psi)}{\partial\psi}\right|\left|\frac{\partial~c(r;\psi)}{\partial\psi'}\right|\nonumber\\
                             &+&2~g(r)\left|c_{n}(r)-c(r;\psi)\right|\left|\frac{\partial^2~c(r;\psi)}{\partial\psi\partial\psi'}\right|\nonumber\\                 &\leq&2~g(r)\left|\frac{\partial~c(r;\psi)}{\partial\psi}\right| \left|\frac{\partial~c(r;\psi)}{\partial\psi'}\right|+4g(r)\left|\frac{\partial^2~c(r;\psi)}{\partial\psi\partial\psi'}\right|\nonumber\\
     &\leq&2~g(r)\left|\frac{\partial~U(r;\psi)}{\partial\psi}\right| \left|\frac{\partial~U(r;\psi)}{\partial\psi'}\right|\nonumber\\
                  &+& 4~g(r)\left[\left|\frac{\partial~U(r;\psi)}{\partial\psi}\right| \left|\frac{\partial~U(r;\psi)}{\partial\psi'}\right| + \left|\frac{\partial^2~U(r;\psi)}{\partial\psi\partial\psi'}\right|\right]\nonumber\\
    &\leq& 6~g(r)\left|\frac{\partial~U(r;\psi)}{\partial\psi}\right| \left|\frac{\partial~U(r;\psi)}{\partial\psi'}\right| + 4~g(r)\left|\frac{\partial^2~U(r;\psi)}{\partial\psi\partial\psi'}\right|.\nonumber
\end{eqnarray}
$\bullet$ {\bf Let us show that $I_n(\psi)$ is of class $\mathcal{C}^2$ with respect to $\alpha$}.\\
According to the previous calculations, we have:
$$\left|\frac{\partial^2}{\partial\alpha\partial\alpha'}f\left(r;\psi\right)\right|\leq6~g(r)\left[\Gamma_2(r)\right]^2 + 4~g(r)\left|\frac{\partial^2~U(r;\psi)}{\partial\alpha\partial\alpha'}\right|.$$
Using the same remarks before, we can show that $U(r;\ \psi)$ is twice differentiable with respect to $\alpha$, and we have:
\begin{eqnarray}
\left|\frac{\partial^2~U(r;\psi)}{\partial\alpha\partial\alpha'}\right|&=& \left|\frac{\partial^2}{\partial\alpha\partial\alpha'}S(r;\ \psi) + \frac{\partial}{\partial\alpha'}\sum_{l=2}^{m+1}\left|\sum_{h=0}^{m+1-l}r_{h+l}c_{h}\right|^{\alpha'}\right|\nonumber\\
&\leq& M~(m+1)^3~H_3(r) + \sum_{l=2}^{m+1}\left[ \left|\sum_{h=0}^{m+1-l}r_{l+h}c_{h}\right|^{\alpha' + 2} \right]\nonumber\\
&\leq& M~(m+1)^3~H_3(r) + \sum_{l=2}^{m+1} H_3(r;l):=\Gamma_3(r),\nonumber
\end{eqnarray}
where
\begin{equation}
H_3(r)=\left\{\begin{array}{ll}
\left(\displaystyle\sum_{l=0}^{m}|r_{l+1}|\right)^{4},\qquad \textrm{si}\ \displaystyle\sum_{l=0}^{m}|r_{l+1}|>1\\\nonumber\\
\left(\displaystyle\sum_{l=0}^{m}|r_{l+1}|\right)^3,\qquad \textrm{si}\ \displaystyle\sum_{l=0}^{m}|r_{l+1}|\leq 1.\\\nonumber\\
\end{array}\right.
\end{equation}
and
\begin{equation}
 H_3(r;l)=\left\{\begin{array}{ll}
\left(\displaystyle\sum_{h=0}^{m+1-l}|r_{h+l}|\right)^{4},\qquad \textrm{si}\ \displaystyle\sum_{h=0}^{m+1-l}|r_{h+l}|>1\\\nonumber\\
\left(\displaystyle\sum_{h=0}^{m+1-l}|r_{h+l}|\right)^3,\qquad \textrm{si}\ \displaystyle\sum_{h=0}^{m+1-l}|r_{h+l}|\leq 1.\\\nonumber\\
\end{array}\right.
\end{equation}
Thus
$$\left|\frac{\partial^2}{\partial\alpha\partial\alpha'}f\left(r;\psi\right)\right|\leq\left[6~\Gamma_2^2(r)+4~\Gamma_3(r)\right]g(r).$$
According to Assumption \ref{Assump3.4}, we get:
$$\int\ldots\int \left[6~\Gamma_2^2(r)+4~\Gamma_3(r)\right]g(r)\ dr<\infty$$
and then $I_n(\psi)$ is of class $\mathcal{C}^2$ with respect to $\alpha$.\\\\
$\bullet$ {\bf Let us show that $I_n(\psi)$ is of class $\mathcal{C}^2$ with respect to $\delta$}.\\
Using the same approach as before we have:
$$\left|\frac{\partial^2}{\partial\delta\partial\delta'}f\left(r;\psi\right)\right|\leq6~g(r)\Gamma_1^2(r) + 4~g(r)\left|\frac{\partial^2~U(r;\psi)}{\partial\delta\partial\delta'}\right|.$$
Since $U(r;\ \psi)$ is twice differentiable with respect to $\delta$, we get:
\begin{eqnarray}
\left|\frac{\partial^2~U(r;\psi)}{\partial\delta\partial\delta'}\right|&=&\frac{\partial^2~S(r;\psi)}{\partial\delta\partial\delta'}+\alpha\frac{\partial}{\partial\delta'}\sum_{l=2}^{m+1}\left|\sum_{h=0}^{m+1-l}r_{h+l}c_{h}\right|^{\alpha-1}\left|\sum_{h=0}^{m+1-l}r_{h+l}\frac{\partial~c_h}{\partial\delta}\right|\nonumber\\
&\leq&\sum_{j=0}^{\infty}\left|\frac{\partial^2~u_j(r;\psi)}{\partial\delta\partial\delta'}\right|+\alpha\sum_{l=2}^{m+1}\left[\left|\sum_{h=0}^{m+1-l}r_{l+h}c_{h}\right|^{\alpha-1}\left|\sum_{h=0}^{m+1-l}r_{l+h}\frac{\partial^2~c_{h}}{\partial\delta\partial\delta'}\right|\right]\nonumber\\
&+&\alpha(\alpha-1)\sum_{l=2}^{m+1}\left[\left|\sum_{h=0}^{m+1-l}r_{l+h}c_{h}\right|^{\alpha-2}\left|\sum_{h=0}^{m+1-l}r_{l+h}\frac{\partial~c_{h}}{\partial\delta}\right|\left|\sum_{h=0}^{m+1-l}r_{l+h}\frac{\partial~c_{h}}{\partial\delta'}\right|\right]\nonumber\\
&\leq&\frac{M~(m+1)}{2}\left(\sum_{l=0}^{m}r_{l+1}\right)^{\alpha}+ B_1 +  B_2\nonumber
\end{eqnarray}
Using the same approach as before, we can compute $B_1$ and $B_2$ as follow:
\begin{eqnarray}
B_1&=&\alpha\sum_{l=2}^{m+1}\left[\left|\sum_{h=0}^{m+1-l}r_{l+h}c_{h}\right|^{\alpha-1}\left|\sum_{h=0}^{m+1-l}r_{l+h}b_h\right|\right]\nonumber\\
&\leq& M_1\sum_{l=2}^{m+1}\left(\sum_{h=0}^{m+1-l}|r_{h+l}|\right)^{\alpha}\nonumber
\end{eqnarray}
and
$$ B_2\leq M_2\sum_{l=2}^{m+1}\left(\sum_{h=0}^{m+1-l}|r_{h+l}|\right)^{\alpha}.$$
Finally we get:
\begin{eqnarray}
\left|\frac{\partial^2~U(r;\psi)}{\partial\delta\partial\delta'}\right|&\leq& \frac{M~(m+1)}{2}\left(\sum_{l=0}^{m}r_{l+1}\right)^{\alpha}+(M_1+M_2)\sum_{l=2}^{m+1}\left(\sum_{h=0}^{m+1-l}|r_{h+l}|\right)^{\alpha}\nonumber\\
&\leq& \frac{M~(m+1)}{2}~H_1(r) + (M_1+M_2)\sum_{l=2}^{m+1} H_1(r;l)=\Gamma_1(r),\nonumber
\end{eqnarray}
and then
$$ \left|\frac{\partial^2}{\partial\delta\partial\delta'}f\left(r;\psi\right)\right|\leq6~g(r)\Gamma_1^2(r) + 4~g(r)\Gamma_1(r).$$
According to Assumption \ref{Assump3.4}, we have $\displaystyle\int\ldots\int \left(6\Gamma_1^2(r) + 4\Gamma_1(r)\right)g(r) dr<\infty$. Thus $I_n(\psi)$ is of class $\mathcal{C}^2$ with respect to $\delta$.\\\\
In short, $I_n(\psi)$ is twice continuously differentiable under the integral sign with respect to $\psi$.
\item[$(A_3$):] The sequence $\left(X_t\right)_{t\in\mathbb{Z}}$ is strictly stationary and ergodic.\\
According to Diongue {\it et al.} \cite{DDM}, the process $\left(X_t\right)_{t\in\mathbb{Z}}$ is strictly stationary and has a unique moving average representation given by:
$$X_t = \sum_{j=0}^{\infty} c_j Z_{t-j},$$
where the sequence $\left(Z_t\right)_{t\in\mathbb{Z}}$ is independently and identically distributed and $\sum_{j=0}^{\infty}\left|c_j\right|^{\alpha}<\infty$, for all $1<\alpha\leq 2$. Hence $\left(X_t\right)_{t\in\mathbb{Z}}$ is an ergodic process. Consequently, the sequence $\left(X_t\right)_{t\in\mathbb{Z}}$ is strictly stationary and ergodic.
\item[$(A_5$):] $K(x; \psi)$ is a measurable function of $x$ for all $\psi$ and bounded, where
$$K(x;\ \psi)=\int\ldots\int\left(\cos(r'x)-c(r;\psi)\right)\frac{\partial}{\partial\psi}~c(r;\psi) g(r)\ dr.$$
We get:
\begin{eqnarray}
\left|\left(\cos(r'x)-c(r;\psi)\right)\frac{\partial}{\partial\psi}~c(r;\psi) g(r)\right|&\leq& 2g(r)\left|\frac{\partial}{\partial\psi}~c(r;\psi)\right|\nonumber\\
&\leq& 2g(r)\left|\frac{\partial}{\partial\psi}~U(r;\psi)\right|.\nonumber
\end{eqnarray}
Using the previous remarks, we obtain:
\begin{equation}
\left|\frac{\partial}{\partial\psi}~U(r;\psi)\right|\leq\Gamma(r)=\left\{\begin{array}{ll}
\Gamma_1(r),\qquad \textrm{si}\ \psi=(d,\ D,\ \phi,\ \theta,\ \Phi,\ \Theta)\\\nonumber\\
\Gamma_2(r),\qquad \textrm{si}\ \psi=\alpha.\\
\end{array}\right.
\end{equation}
Thus
$$\left|\left(\cos(r'x)-c(r;\psi)\right)\frac{\partial}{\partial\psi}~c(r;\psi) g(r)\right|\leq 2\Gamma(r)g(r),\quad\forall~\psi\in\Psi.$$
Using Assumption \ref{Assump3.4}, we have $\displaystyle\int\ldots\int \Gamma(r)g(r)\ dr<\infty$, and hence $K(x;\psi)$ is a measurable function of $x$ for all $\psi$ and bounded.
\item[$(A_6$):] $B(\psi_0)=\displaystyle\int\ldots\int\frac{\partial}{\partial\psi}~c(r;\psi_0)\frac{\partial}{\partial\psi'}~c(r;\psi_0)g(r)\ dr$ is nonsingular and $\displaystyle\frac{\partial^2}{\partial\psi\partial\psi'}~c(r;\psi)$ is uniformly bounded by a g-integrable function over $\Psi$.\\\\
We get:
$$\left|\frac{\partial}{\partial\psi}~c(r;\psi_0)\frac{\partial}{\partial\psi'}~c(r;\psi_0)\right|\leq \left|\frac{\partial}{\partial\psi}~U(r;\psi_0)\frac{\partial}{\partial\psi'}~U(r;\psi_0)\right|\leq \Gamma^2(r)$$
According to Assumption \ref{Assump3.4}, we have $\displaystyle\int\ldots\int \Gamma^2(r)g(r)\ dr<\infty$, and then
$$\displaystyle\int\ldots\int \left|\frac{\partial}{\partial\psi}~c(r;\psi_0)\frac{\partial}{\partial\psi'}~c(r;\psi_0)\right|g(r)\ dr<\infty.$$
Hence $B(\psi_0)$ is nonsingular. On the other hand, remark that:
\begin{eqnarray}
\left|\frac{\partial^2~c(r;\psi)}{\partial\psi\partial\psi'}\right|&\leq& \left|\frac{\partial~U(r;\psi)}{\partial\psi}\right| \left|\frac{\partial~U(r;\psi)}{\partial\psi'}\right| + \left|\frac{\partial^2~U(r;\psi)}{\partial\psi\partial\psi'}\right|\nonumber\\
&\leq& \Gamma^2(r)+\left|\frac{\partial^2~U(r;\psi)}{\partial\psi\partial\psi'}\right|.\nonumber
\end{eqnarray}
Now we have:
\begin{equation}
\left|\frac{\partial^2~U(r;\psi)}{\partial\psi\partial\psi'}\right|<\Gamma'(r)=\left\{\begin{array}{ll}
\Gamma_1(r),\qquad \textrm{si}\ \psi=(d,\ D,\ \phi,\ \theta,\ \Phi,\ \Theta)\\\nonumber\\
\Gamma_3(r),\qquad \textrm{si}\ \psi=\alpha.\\
\end{array}\right.
\end{equation}
This implies that: $$\left|\frac{\partial^2~c(r;\psi)}{\partial\psi\partial\psi'}\right|\leq \Gamma^2(r)+\Gamma'(r),$$
and using Assumption \ref{Assump3.4}, we get:
$$\displaystyle\int\ldots\int (\Gamma^2(r)+\Gamma'(r))g(r)\ dr<\infty.$$
Hence $\left|\displaystyle\frac{\partial^2~c(r;\psi)}{\partial\psi\partial\psi'}\right|$ is uniformly bounded by a g-integrable function over $\Psi$.
\end{enumerate}
Thus all conditions are verified, then we can apply the Theorem 2.1 of Knight and Yu \cite{KY} to achieve the proof.
\hspace{\stretch{1}}\rule{1ex}{1ex}


\begin{thebibliography}{00}
\bibitem{AO2005} A. Alvarez and P. Olivares. M\'ethodes d'estimation pour des lois stables avec des applications en finance. Journal de la Soci\'et\'e Fran\c caise de Statistique, (1)4, (2005) 23-54. 
\bibitem{BMN2007}M. Boutahar, V. Marimoutou, L. Nouira, Estimation Methods of the Long Memory
Parameter: Monte Carlo Analysis and Application. Journal of Applied Statistics, 34(3), (2007), 261-301.
\bibitem{BD2006} P. J. Brockwell and R.A. Davis, Time Series: Theory and Methods, 2nd ed., in: Springer Series in Statistics, 2006
\bibitem{CD} M. Calder and R.A. Davis, Inferences for linear processes with stable noise. In R.J. Adler,R.E. Feldman, and M.S. Taqqu (eds.), A Practical Guide to Heavy Tails, pp. 159-176. Boston:Birkhiuser (1998).
\bibitem{chaMaSt}J. M. Chambers, C. Mallows and B. W. Stuck, A method for simulating stable random variables. Journal of the American Statistical Association, 71(354), (1976) 340-344. Theory and Methods Section.
\bibitem{CCFG2007} M. Carrasco, M. Chernov, J. P. Florens and E. Ghysels, Efficient estimation of general dynamic models with a continuum of moment conditions, 140(2007), 529-573.
\bibitem{Cli} D.B.H. Cline, Estimation and linear prediction for regression, autoregression and ARMA with infinite variance data,Ph.D.Dissertation, Statistics Department, Colorado State University, 1983.
\bibitem{CliBroc}  D.B.H. Cline and P.J. Brockwell, Linear prediction of ARMA processes with infinite variance. Stochastic Processes and their applications, 19, (1985) 281-296.
\bibitem{DG2008}A. K. Diongue and D. Gu\'egan, Estimation of $k$-Factor GIGARCH Process: A Monte Carlo Study. Communications in Statistics—Simulation and Computation, 37(2008), 2037-2049.
\bibitem{DDM}A.~K. Diongue, A. Diop and M. Ndongo, Seasonal fractional ARIMA with stable innovations. Statistics and Probability Letters, 78, (2008) 1404-1411.
\bibitem{Fama} E. Fama, Behavior of stock market prices, Journal of Business 38 (1) (1965) 34-105.
\bibitem{GL}L. Giraitis and R. Leipus, A generalized fractionally differencing approach in long memory modelling. Lithuanian Mathematical Journal, 35, (1995) 65-81.
\bibitem{Heathcote1977} C. R. Heathcote, The integrated squared error estimation of parameters. Biometrik, 64, (1977), 255-264.
\bibitem{KY}  J. L. Knight and J. Yu, The empirical characteristic function in time series estimation. Econometric Theory 18, (2002) 691-721.
\bibitem{KT1995}P. S. Kokoszka and M. S. Taqqu, Fractional ARIMA with stable innovations. Stochastic Processes and their Applications, 60, (1995) 19-47.
\bibitem{NDDS} M. Ndongo, A. K. Diongue, A. Diop, S. Dossou-Gb?t?, Estimation of long-memory parameters for seasonal fractional ARIMA with stable innovations. Statistical Methodology, 7 (2010) 141-151.
\bibitem{Ndongo2011} M. Ndongo, Les processus \`a m\'emoire longue saisonniers avec variance infinie des innovations et leurs applications, Th\`ese de Doctorat unique, Universit\'e Gaston Berger de Saint-Louis (S\'en\'egal), 2011.
\bibitem{NMc1994} W.K. Newey and D. McFadden, Large sample estimation and hypothesis testing. In R.F. Engle and D. McFadden (eds.), Handbook of Econometrics vol. 4. Amsterdam: North-Holland.
\bibitem{PHL}  A. S. Paulson,  E. W. Holcomb, R. A. Leitch. The estimation of the parameters of the stable laws. Biometrika 62, (1975) 163-170.
\bibitem{Rprog} R Development Core Team  R: A language and environment for statistical computing. R Foundation for Statistical Computing, Vienna, Austria. ISBN 3-900051-07-0, URL http://www.R-project.org, 2008.
\bibitem{RRP1}V. A. Reisen, A. L. Rodrigues and W. Palma, Estimation of seasonal fractionally integrated processes. Computational Statistics \& Data Analysis, 50, (2006) 568-582.
\bibitem{RC} C. P Robert and  G. Casella,  Monte Carlo Statistic Methods, 2nd ed. ,Springer Texts in Statistics, 2004.
\bibitem{SamTaqqu}G. Samorodnitsky, and M. Taqqu, Stable non-gaussian random processes : Stochastic models with infinite variance, Chapman \& Hall, 1994.
\bibitem{stotaq}S. Stoev, and M. Taqqu, Simulation methods for linear fractional stable motion and FARIMA using the Fast Fourier Transform. Fractals, 12, Nb 1, (2004) 95-121.
\bibitem{SK1974} B.W. Stuck and B. Kleiner, A statistical analysis of telephone noise, The Bell System Technical Journal 53 (1974) 1263-1320.
\bibitem{yu} J. Yu, Empirical characteristic function estimation and its applications, Econometric Reviews, 23, (2004) 93-123.
\end{thebibliography}
\end{document}